\documentclass[11pt, a4paper]{article}
\usepackage[utf8]{inputenc}
\usepackage[T1]{fontenc}
\usepackage[american]{babel}
\usepackage{amsmath}
\usepackage{amsthm}
\usepackage{dsfont}
\usepackage{amssymb}
\usepackage{mathtools}
\usepackage{graphicx}
\usepackage[shortlabels]{enumitem}
\usepackage[colorinlistoftodos,color=red!50!blue!50]{todonotes}
\usepackage{booktabs}
\usepackage{subcaption}
\usepackage{siunitx}
\usepackage[lined,ruled,noend,linesnumbered]{algorithm2e}
\usepackage[hypertexnames=false,breaklinks,colorlinks=true,citecolor=blue,linkcolor=blue,urlcolor=blue!50!black]{hyperref}
\usepackage[tmargin=2.5cm,bmargin=2.5cm,lmargin=2.8cm,rmargin=2.8cm]{geometry}
\usepackage{eurosym}
\usepackage{xspace}

\allowdisplaybreaks
\newcommand{\suchthat}{\,:\,}
\newcommand{\abs}[1]{\lvert{#1}\rvert}
\newcommand{\card}[1]{\lvert{#1}\rvert}

\newcommand{\define}{\coloneqq}

\newcommand{\dcup}{\mathbin{\dot{\cup}}}

\DeclareSIUnit\eurocur{\mbox{\footnotesize\euro{}}}
\DeclareSIUnit\bar{bar} 

\newcommand{\graph}{{G}}
\newcommand{\nodes}{V}
\newcommand{\edges}{A}
\newcommand{\gsup}{{G^S}}
\newcommand{\gret}{{G^R}}
\newcommand{\nsup}{{V^{S}}}
\newcommand{\nret}{{V^{R}}}
\newcommand{\esup}{{A^{S}}}
\newcommand{\eret}{A^{R}}
\newcommand{\epipe}{{A^{P}}}
\newcommand{\edemand}{{A^{D}}}
\newcommand{\esupplier}{{A^{H}}}
\newcommand{\estorage}{{A^{C}}}

\newcommand{\pr}[1][i]{\ensuremath{{p}_{#1}}}
\newcommand{\pfix}{\ensuremath{\overline{p}}}
\newcommand{\mf}[1][ij]{\ensuremath{\dot{m}_{#1}}}
\newcommand{\mfp}[1][ij]{\ensuremath{\dot{m}_{#1}^+}}
\newcommand{\mfn}[1][ij]{\ensuremath{\dot{m}_{#1}^-}}
\newcommand{\mfpos}[1][ij]{\ensuremath{{b}^{+}_{#1}}}
\newcommand{\mfneg}[1][ij]{\ensuremath{{b}^{-}_{#1}}}
\newcommand{\mfposstar}[1][ij]{\ensuremath{{b}^{+,\star}_{#1}}}
\newcommand{\mfnegstar}[1][ij]{\ensuremath{{b}^{-,\star}_{#1}}}
\newcommand{\mfposprime}[1][ij]{\ensuremath{{b}^{+\prime}_{#1}}}
\newcommand{\mfnegprime}[1][ij]{\ensuremath{{b}^{-\prime}_{#1}}}
\newcommand{\heatpower}[1][i]{\ensuremath{\dot{Q}_{#1}}}
\newcommand{\pumppower}[1][i]{\ensuremath{P_{#1}}}
\newcommand{\stpower}[1][i]{\ensuremath{\dot{S}_{#1}}}
\newcommand{\tstart}[1][ij]{\ensuremath{{T^{+}_{#1}}}}
\newcommand{\tstartstar}[1][ij]{\ensuremath{{T^{+,\star}_{#1}}}}
\newcommand{\tstartprime}[1][ij]{\ensuremath{{T^{+\prime}_{#1}}}}
\newcommand{\tstartprimeprime}[1][ij]{\ensuremath{{T^{+\prime\prime}_{#1}}}}
\newcommand{\tend}[1][ij]{\ensuremath{T^{-}_{#1}}}
\newcommand{\tendstar}[1][ij]{\ensuremath{T^{-,\star}_{#1}}}
\newcommand{\tendprime}[1][ij]{\ensuremath{T^{-\prime}_{#1}}}
\newcommand{\tendprimeprime}[1][ij]{\ensuremath{T^{-\prime\prime}_{#1}}}
\newcommand{\tnode}[1][i]{\ensuremath{{T}_{#1}}}

\newcommand{\length}[1][ij]{\ensuremath{{\ell}_{#1}}}
\newcommand{\diam}[1][ij]{\ensuremath{{d}_{#1}}}
\newcommand{\height}[1][i]{\ensuremath{{h}_{#1}}}
\newcommand{\pipefriction}[1][ij]{\ensuremath{{\lambda_{#1}}}}
\newcommand{\mubound}[1][ij]{\ensuremath{{\overline{m}_{#1}}}} 
\newcommand{\tamb}{\ensuremath{\tau}}
\newcommand{\heatloss}[1][ij]{\ensuremath{\mu_{#1}}}
\newcommand{\heatcap}{\ensuremath{{c_p}}}
\newcommand{\pumpeff}[1][i]{\ensuremath{{\eta_{#1}}}}
\newcommand{\heatdemand}[1][i]{\ensuremath{\dot{Q}^{D}_{#1}}}
\newcommand{\tfix}{\ensuremath{\underline{T}}}
\newcommand{\ploss}{\Delta {p}}
\newcommand{\plossmin}{\ensuremath{\underline{\ploss}}}
\newcommand{\tubound}[1]{\ensuremath{\overline{T}_{#1}}}
\newcommand{\powerbound}[1][i]{\ensuremath{{\underline{\dot{Q}}_{#1}}}}
\newcommand{\stcap}[1]{\overline{S}_{#1}}
\newcommand{\heatcost}[1][i]{\ensuremath{\dot{c}_{#1}}}
\newcommand{\eleccost}[1][i]{\ensuremath{c_{#1}}}
\newcommand{\roughness}[1][ij]{\ensuremath{{\epsilon}_{#1}}}

\newcommand{\degreeCelsius}{\si{\celsius}}
\newcommand{\kgs}{\si[per-mode=fraction]{\kg\per\second}}
\newcommand{\Pascal}{\si{\pascal}}
\newcommand{\kW}{\si{\kilo\watt}}
\newcommand{\meter}{\si{\meter}}
\newcommand{\kwmc}{\si[per-mode=fraction]{\kilo\watt\per\meter\per\celsius}}
\newcommand{\kgm}{\si[per-mode=fraction]{\kilogram\per\cubic\meter}}
\newcommand{\mss}{\si[per-mode=fraction]{\meter\per\second\squared}}
\newcommand{\fracmkg}{\si[per-mode=fraction]{\per\meter\per\kilogram}}
\newcommand{\kjkgc}{\si[per-mode=fraction]{\kilo\joule\per\kilogram\per\celsius}}
\newcommand{\euromwh}{\si[per-mode=fraction]{\eurocur\per\mega\watt\per\hour}}

\newtheorem{theorem}{Theorem}
\newtheorem{lemma}[theorem]{Lemma}
\newtheorem{corollary}[theorem]{Corollary}

\newtheorem{remark}[theorem]{Remark}
\newtheorem{assumption}[theorem]{Assumption}

\setlist[itemize]{leftmargin=3ex,topsep=0.5ex,partopsep=0ex,parsep=0ex,itemsep=0.5ex}

\title{Global Optimization of Flexible District Heating Networks}
\author{Marc E. Pfetsch\thanks{Department of Mathematics, TU Darmstadt,
    Germany} \and Lea Rehlich \and Florian Steinke\thanks{Energy Information Networks \&
Systems Lab, TU Darmstadt, Germany} \and Stefan Ulbrich$^*$}
\date{July 2025}

\begin{document}

\maketitle

\begin{abstract}
  \noindent
  District heating networks are a central tool to achieve low-carbon heat supplies.
  In this realm, they face the challenge of dealing with increasingly heterogeneous, partially time-varying renewable sources, thermal storage, and meshed topologies.
  This paper examines global optimization of the operation of such district heating networks over multiple time steps, based on a stationary, yet realistical nonlinear network model.

  To accelerate the solution performance of a spatial branch-and-bound algorithm for one time step, the following new methodological ingredients are introduced:
  exclusion of cyclic flow, symmetry exploitation between supply and return networks, reduction of temperature mixing constraints, novel primal heuristics and branching rules.
  The proposed methods are evaluated on a set of generated and real-world benchmark network instances with cycles and several suppliers.
  On the generated benchmark instances, using these methods more than doubles the number of solved instances and more than halves the runtime.
  For the real-world benchmark instances, the resulting algorithm produces solutions with guaranteed quality in reasonable run time.

  For multiple time steps that are coupled by a storage, a time decomposition approach is investigated.
  Under assumptions that are reasonable in practice, this approach is shown to yield an optimal solution.
  On a small example network, this decomposition is able to compute optimal solutions in less than a second, while solving the complete time-coupled problem is not possible within one hour.
\end{abstract}

\section{Introduction}

The decarbonization of energy systems in order to limit climate change foresees in many regions a central role for district heating networks (DHNs), e.g., for the European Union \cite{renewable_heating_cooling_pathways_2023}.
These DHNs will increasingly integrate renewable heat sources, such as geothermal energy, waste heat from industrial processes, or time-variable renewable electricity from wind and sun converted into heat via heat pumps.
The multitude of different and partly fluctuating heat sources motivates the integration of additional thermal energy storage units into DHNs.
Also, the supply temperatures in DHNs are expected to decrease in order to reduce distribution losses and to allow for better efficiency of heat pumps. 
Increased flexibility demands additionally favor grid topologies beyond traditional, radial structures, allowing for multiple flow paths and changing flow directions in the network.
Such novel, complex DHNs systems are often referred to as 4th generation DHNs \cite{LUND20141}.
Their optimal operation is central for cost-efficient and environmentally friendly future heat supplies and is addressed in this paper. 


Various prior works have addressed the optimal operation of DHNs, see \cite{li2017district,TalebiMirzaeiBastaniHaghighat2016} for reviews.
Some approaches consider dynamic flow models of DHNs, e.g., \cite{nodhn,Dnschel2022AdaptiveNO,HeringCansevTamassiaXhonneuxMueller2021,HeringFallerXhonneauxMueller2022,JaekleReichleVolkwein}, which are able to capture transient effects in the network, but are computationally very challenging, since they require to discretize both in space and time.
In contrast, stationary DHN models assume steady-state conditions in the network at each of a few considered time steps and focus on the nodal values of pressure, temperature, and mass flow rate, which are the relevant physical quantities in DHNs.
Such models are sufficiently accurate for not too large networks and sufficiently long operational planning time steps \cite{bott2023deep}.
They render the scheduling problem much more tractable. However, the corresponding optimization problems are still non-convex and challenging.
To further reduce problem size and complexity, aggregation methods combine multiple consumers or suppliers into single equivalent nodes \cite{Loewen_2001, Larsen_2004} and multiple pipe segments into single equivalent pipes \cite{BOTT21}.
The resulting optimization problems have then often been addressed via linearization and mixed-integer linear programming (MILP), see, e.g., \cite{soderman2007optimisation}, via local search methods \cite{savola2007minlp} or via heuristics \cite{yiqing2007improved}.
Finding global optima of the exact, nonlinear system has not been examined widely.
If so, very small instances have been considered only, e.g., in the expansion context \cite{dhnExpansion}.

This paper presents a global optimization approach for the operation of DHNs with diverse heat sources, thermal energy storage units, and flexible flow directions.
It is based on a stationary, but realistical nonlinear DHN model.
Five new methodological ingredients are developed to tackle the arising challenging mixed-integer nonlinear programming (MINLP) problems:
(i) exclusion of cyclic flow,
(ii) the exploitation of symmetries in the supply and return network,
(iii) a reduction of temperature mixing constraints,
(iv) two primal heuristics, and
(iv) branching rules that fix the point in the network where the smallest pressure difference between supply and return part is attained.
For a single time step, the developed methods are evaluated on a set of benchmark instances including both generated as well as real-world inspired DHN topologies.
Moreover, for settings with multiple time steps that are coupled through a storage, a decomposition algorithm is developed.
The optimality of this approach is proved under certain assumptions.
While these assumptions are not fully general, they still cover important practically relevant use cases.

For a single time step and the test set of generated instances, the proposed algorithm improvements allow to more than double the number of solved instances and more than halve the solution time compared to the baseline spatial branch-and-bound implementation.
For the real-world instances, it allows to compute solutions with a guaranteed quality within reasonable time. 
Moreover, for a plausible four time step scheduling instance including a storage, the decomposition approach enables solvability in the first place, and that even in execution times below a second,
while the time-coupled problem cannot be solved within one hour.
Unlike MILP-based or heuristic approaches, this MINLP approach additionally provides accurate valid optimality guarantees, which is important in practical operational planning for solution stability and predictability.

This paper is structured as follows. Section \ref{sec:Modeling} presents the mathematical model for scheduling stationary DHNs, first for a single time step.
Properties of the resulting non-convex MINLP are discussed in Section \ref{sec:Properties}, including cycle flow and supply/return network symmetry.
Optimization problems with one time step are then experimentally evaluated in Section \ref{sec:NumericalResultsI}.
In Section \ref{sec:Storages}, we extend the model to multiple time steps and thermal energy storage units.
The novel decomposition approach to solve the corresponding optimization problems is presented and evaluated in Section \ref{sec:StorageOpt}.
Finally, Section \ref{sec:Conclusion} concludes the paper and gives an outlook on future research directions.

\section{Mathematical Modeling of Stationary Heating Networks}
\label{sec:Modeling}

In this section we present a mathematical optimization model that allows to
optimize the stationary operation of heating networks. The model is similar
to the one of K\"ocher \cite{koecher}, but there are multiple articles with similar models, for
example \cite{nodhn, borsche, en12071215, DUQUETTE2016383, LIU20161238}.

The main components of a district heating network are suppliers, consumers,
and pipes. Suppliers heat the water that flows through the pipes of the
network and control the pressure via pumps. Thus, they control the injected
mass flow rate of the water and its temperature. The consumers require an
amount of heat that has to be satisfied. The following stationary
mathematical model represents this behavior.

\subsection{Heating Networks}

We represent a district heating network as a (weakly) connected directed
graph $\graph = (\nodes, \edges)$ with nodes $\nodes$ and arcs
$\edges \subset V \times V$. We can assume that $\graph$ is simple, i.e.,
there are neither self-loops nor (anti)parallel arcs.

\begin{table}[bt]
  \caption{Model sets.}
  \label{tab:sets}
  \centering
  \begin{tabular}{@{}cl@{}}
    \toprule
    Set & Description \\
    \midrule
    $\graph = (\nodes, \edges)$ & network graph with nodes $\nodes$ and arcs $\edges \subset \nodes \times \nodes$\\
    $\epipe \subset \edges$ & pipe arcs \\
    $\edemand \subset \edges$ & demand arcs representing consumers \\
    $\esupplier \subset \edges$ & heating arcs representing suppliers \\
    $\nsup \subset \nodes$ & nodes in the supply part of the network \\
    $\nret \subset \nodes$ & nodes in the return part of the network \\
    $\esup \subset \epipe$ & pipe arcs in the supply part of the network \\
    $\eret \subset \epipe$ & pipe arcs in the return part of the network \\
    $\gsup = (\nsup, \esup)$ & supply network graph \\
    $\gret = (\nret, \eret)$ & return network graph \\
    \bottomrule
  \end{tabular}
\end{table}

A heating network consists of the supply and return
part. Accordingly, the set of nodes~$\nodes$ is split into \emph{supply
  nodes} $\nsup$ and \emph{return nodes} $\nret$. In each part, the nodes
are connected by \emph{supply} and \emph{return pipe arcs} $\esup$ and
$\eret$, respectively. The two parts are connected by \emph{heating arcs}
$\esupplier$ and \emph{demand arcs} $\edemand$, representing suppliers and
consumers, respectively. Hot water flows from the suppliers to the
consumers through pipes in $\esup$. The consumers at demand arcs $\edemand$
extract energy, cooling down the water, which is then inserted into the
return part~$\eret$, where it flows back to the suppliers on the heating
arcs $\esupplier$. There, the water taken from the return part, is heated
and reinserted into the supply part. Thus, the graph
$\graph = (\nodes, \edges)$ consists of nodes $\nodes = \nsup \dcup \nret$
and arcs $\edges = \esup \dcup \eret \dcup \esupplier \dcup \edemand$.
The \emph{pipe arcs} $\epipe \define \esup \dcup \eret$ yield
two disjoint graphs, the \emph{supply network graph}
$\gsup = (\nsup, \esup)$ and the \emph{return network graph}
$\gret = (\nret, \eret)$. The sets of the network model are summarized in
Table \ref{tab:sets}. Figure \ref{fig:modell1} shows a small example
network with $\abs{\nodes} = 6$ nodes and $\abs{\edges} = 7$ arcs.

We assume that heating arcs are always directed from the return to the
supply network, while demand arcs are always directed from the supply to
the return network, i.e., $\esupplier \subset \nret \times \nsup$ and
$\edemand \subset \nsup \times \nret$. Moreover, for simplicity, we assume
that every node is connected either to a heating arc $\esupplier$ or a
demand arc $\edemand$. In reality, this may not hold, but then either multiple
pipes could be combined into one, see, e.g.,~\cite{BOTT21}, or demand arcs
with a demand of 0 can be added.

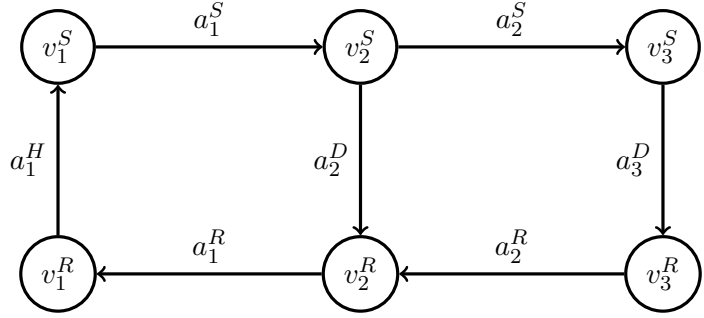
\begin{figure}[tb]
  \centering
  \begin{tikzpicture}
    \node at (0,3) [circle,draw, very thick] (va){$v^S_1$};
    \node[very thick] at (4,3) [circle,draw] (vb){$v^S_2$};
    \node at (8,3) [circle,draw, very thick] (vc){$v^S_3$};
    \node at (0,0) [circle,draw, very thick] (ra){$v^R_1$};
    \node at (4,0) [circle,draw, very thick] (rb){$v^R_2$};
    \node at (8,0) [circle,draw, very thick] (rc){$v^R_3$};
    \draw[->, very thick] (ra)-- node[left]{$a^{H}_{1}$}(va);
    \draw[->, very thick] (va)-- node[above]{$a^{S}_1$}(vb);
    \draw[->, very thick] (vb)-- node[above]{$a^{S}_2$}(vc);
    \draw[->, very thick] (vc)-- node[left]{$a^{D}_{3}$}(rc);
    \draw[->, very thick] (rc)-- node[above]{$a^{R}_2$}(rb);
    \draw[->, very thick] (rb)-- node[above]{$a^{R}_1$}(ra);
    \draw[->, very thick] (vb)-- node[left]{$a^{D}_{2}$}(rb);
  \end{tikzpicture}
  \caption{Small example network with $\nsup = \{v^S_1, v^S_2, v^S_3\}$ and
    $\esup = \{a^{S}_1, a^{S}_2\}$ as well as
    $\nret = \{v^R_1, v^R_2, v^R_3\}$ and $\eret = \{a^{R}_1, a^{R}_2\}$,
    which specify the supply and return part, respectively. Moreover,
    $a^{H}_1$ is a heating arc, while $a^{D}_{2}$, $a^{D}_{3}$ are demand
    arcs.}
  \label{fig:modell1}
\end{figure}

Important state variables are the mass flow rate $\mf[a]$ in \kgs\ for an
arc $a \in \edges$ and the pressure $\pr[v]$ in \Pascal\ at a node
$v \in \nodes$. Moreover, we have the temperature $\tstart[a]$ and
$\tend[a]$ in \degreeCelsius\ at the start and end of arc~$a$ (in flow direction),
respectively, as well as the temperature $\tnode[v]$ in~$\degreeCelsius$ at
a node $v$. Table~\ref{tab:vars} gives a list of the used variables and
Table~\ref{tab:parameters} the used parameters. In the first part of this
paper, we consider the stationary case such that all variables are constant
over time.

\begin{table}[tb]
  \caption{Model variables.}
  \label{tab:vars}
  \centering
  \begin{tabular}{@{}cll@{}} \toprule 
    Variable & Description & Unit \\
    \midrule
    \mf[a] & mass flow rate on arc $a \in \edges$ & \kgs \\
    \pr[v] & pressure at node $v \in \nodes$ & \Pascal \\
    \tnode[v] & temperature at node $v \in \nodes$ & \degreeCelsius \\
    \tstart[a] & start temperature of the flow on arc $a \in \edges$ (in flow direction) & \degreeCelsius \\
    \tend[a] & end temperature of the flow on arc $a \in \edges$  (in flow direction) & \degreeCelsius \\
    \mfpos[a] & binary positive flow direction variable on arc $a \in \edges$ & -- \\
    \mfneg[a] & binary negative flow direction variable on arc $a \in \edges$ & -- \\
    \heatpower[a] & heat energy of the heating arc $a \in \esupplier$ & \kW \\
    \pumppower[a] & electric power of the heating arc $a \in \esupplier$ & \kW \\
    \bottomrule
  \end{tabular}
\end{table} 

\begin{table}[tb]
  \caption{Model parameters.}
  \label{tab:parameters}
  \centering
  \begin{tabular}{@{}cllcc@{}} 
    \toprule
    Parameter & Description & Indices & Unit & Value \\
    \midrule
    $\zeta_a$ & pipe dependent parameters & $a \in \epipe$ & \fracmkg & \\
    \length[a] & length of a pipe & $a \in \epipe$ & \meter & \\
    \diam[a] & diameter of a pipe & $a \in \epipe$ & \meter & \\
    \heatloss[a] & heat loss in a pipe & $a \in \epipe$ & \kwmc & \\
    \pipefriction[a] & friction-coefficient of a pipe & $a \in \epipe$ & -- & \\
    \roughness[a] & roughness of a pipe & $a \in \epipe$ & \meter & \\ %
    \tubound{a} & upper temperature bound at end of arc & $a \in \epipe$ & \degreeCelsius \\
    \powerbound[a] & lower energy bound & $a \in \esupplier$ & \kW \\
    \pumpeff[a] & pump efficiency & $a \in \esupplier$ & -- & \\
    \heatdemand[a] & given heat demand & $a \in \edemand$ & \kW & \\
    \plossmin & minimal pressure drop & $a \in \edemand$ & \Pascal & $10^4$ \\
    \height[v] & height of node & $v \in \nodes$ & \meter & \\
    \heatcap & specific heat capacity of water & -- & \kjkgc & $4.184$ \\
    \tamb & ambient temperature & -- & \degreeCelsius & $14$ \\
    $\rho$ & density of water & -- & \kgm & 997 \\
    $g$ & gravitational constant & -- & \mss & 9.81 \\
    \pfix & fixed pressure at one node & -- & \Pascal & $10^6$  \\
    \tfix & fixed temperature at end of demand arcs & -- & \degreeCelsius & $40$ \\
    \bottomrule
  \end{tabular}
\end{table}

\subsection{Hydraulic Variables and Constraints}
\label{sec:HydraulicModel}

Every arc $a \in \edges$ has a corresponding flow rate $\mf[a]$. A positive
value indicates flow along the direction of~$a$, while negative values
indicate flow against its direction. For heating and demand arcs
$a \in \esupplier \cup \edemand$, we always assume that $\mf[a] \geq 0$.

\paragraph{Hydraulic Behavior on Pipes} On pipes, we use the following
model providing an approximation of the hydraulic behavior on a pipe
$a = (u,v) \in \epipe$ (see, e.g., \cite{koecher}):
\[
  \pr[u] - \pr[v] = \zeta_a \, \mf[a]\, \abs{\mf[a]} - \rho\, g\, (\height[u] -
  \height[v]),
\]
where $\rho$ is the density of water in \kgm, $g$ the gravitational constant in \mss, and
$\height[v]$ is the \emph{height} in \meter\ of a node $v \in \nodes$. Moreover,
\begin{align*}
  \zeta_a \define \pipefriction[a] \frac{\length[a]}{\diam[a]^5} \frac{8}{\pi^2\, \rho}
\end{align*}
in \fracmkg depends on the following pipe parameters:
$\pipefriction[a]$ is the dimensionless \emph{friction coefficient},
$\length[a]$ and~$\diam[a]$ are the \emph{length} and
\emph{diameter} of pipe $a$ in \meter, respectively. The above
constraint follows from the stationary momentum equation and the
incompressibility of water within a pipe. For a detailed derivation of
this constraint from the nonlinear Euler equations we refer to, e.g.,
Roland and Schmidt~\cite{dhnExpansion}.

To model the friction coefficient, which represents the pressure loss in a
pipe due to friction, we use the formula of
Nikuradse~\cite{Nikuradse1933,Nikuradse1950}
\begin{equation}\label{eq:Nikuradse}
  \pipefriction[a] = \left(2 \log_{10}\left(\frac{ \diam[a]
  }{\roughness[a]}\right) + 1.138\right)^{-2}.
\end{equation}
This formula depends on the roughness $\roughness[a]$ of the pipe~$a$ and
the diameter $\diam[a]$, but it is independent of the mass flow rate.

An alternative choice for modeling the friction factor would be to use
the Colebrook-White formula~\cite{colebrook}, where the friction is
the solution of a nonlinear equation. Because of its increased
complexity with limited additional accuracy, in this paper we
use~\eqref{eq:Nikuradse}. See Kazda and Li~\cite{KazdaLi} for a
further comparison of these two friction factors.

\paragraph{Flow Directions}
We focus on networks in which the direction of the flows is not fixed a
priori. Figure \ref{fig:example_bid} shows two example networks, where the
mass flow direction is not uniquely defined. One network has two heating
arcs $a^H_1$ and $a^H_4$, while the other contains a cycle. Arcs in solid blue
have flow direction that is fixed by definition, while the dashed black
arcs do not.

\begin{figure}[tb]
  \centering
  \begin{minipage}[t]{.49\linewidth}
    \centering
    \begin{tikzpicture}[scale=0.7]
      \node at (0,3) [circle,draw, very thick] (va){$v^S_1$};
      \node at (2.5,3) [circle,draw, very thick] (vb){$v^S_2$};
      \node at (5,3) [circle,draw, very thick] (vc){$v^S_3$};
      \node at (7.5,3) [circle,draw, very thick] (vd){$v^S_4$};
      \node at (0,0) [circle,draw, very thick] (ra){$v^R_1$};
      \node at (2.5,0) [circle,draw, very thick] (rb){$v^R_2$};
      \node at (5,0) [circle,draw, very thick] (rc){$v^R_3$};
      \node at (7.5,0) [circle,draw, very thick] (rd){$v^R_4$};
      \draw[->, color=blue, very thick] (ra)-- node[left]{$a^{H}_{1}$}(va);
      \draw[->, color=blue, very thick] (rd)-- node[left]{$a^{H}_{4}$}(vd);
      \draw[<->, dashed, very thick] (va) -- (vb);
      \draw[<->, dashed, very thick] (vb) -- (vc);
      \draw[->, color=blue, very thick] (vc)-- node[left]{$a^{D}_{3}$}(rc);
      \draw[<->, dashed, very thick] (rc) -- (rb);
      \draw[<->, dashed, very thick] (rb) -- (ra);
      \draw[->, color=blue, very thick] (vb)-- node[left]{$a^{D}_{2}$}(rb);
      \draw[<->, dashed, very thick] (vd) -- (vc);
      \draw[<->, dashed, very thick] (rc) -- (rd);
    \end{tikzpicture}
    \subcaption{Example network with two suppliers.}
    \label{fig:regina}
  \end{minipage}
  \hfill%
  \begin{minipage}[t]{.49\linewidth}
    \centering
    \begin{tikzpicture}[scale=0.7]
      \node at (0,3) [circle,draw, very thick] (va){$v^S_1$};
      \node at (2.5,3) [circle,draw, very thick] (vb){$v^S_2$};
      \node at (5,3) [circle,draw, very thick] (vc){$v^S_3$};
      \node at (0,0) [circle,draw, very thick] (ra){$v^R_1$};
      \node at (2.5,0) [circle,draw, very thick] (rb){$v^R_2$};
      \node at (5,0) [circle,draw, very thick] (rc){$v^R_3$};
      \draw[->, color=blue, very thick] (va)-- node[left]{$a^{D}_{1}$}(ra);
      \draw[->, color=blue, very thick] (rb)-- node[left]{$a^{H}_{2}$}(vb);
      \draw[->, color=blue, very thick] (vc)-- node[left]{$a^{D}_{3}$}(rc);
      \draw[<->, dashed, very thick] (va) -- (vb);
      \draw[<->, dashed, very thick] (vb) -- (vc);
      \draw[<->, dashed, very thick] (rb) -- (ra);
      \draw[<->, dashed, very thick] (rc) -- (rb);
      \draw[<->, dashed, very thick] (rc) -- (6.25,0) -- (6.25,-1) -- (-1.25,-1) -- (-1.25,0) -- (ra);
      \draw[<->, dashed, very thick] (vc) -- (7,3) -- (7,-1.5) -- (-2,-1.5) -- (-2,3) -- (va);
    \end{tikzpicture}
    \subcaption{Example cyclic network.}
    \label{fig:circulus}
  \end{minipage}
  \caption{Example networks with some unknown mass flow directions (solid blue).}
  \label{fig:example_bid}
\end{figure}
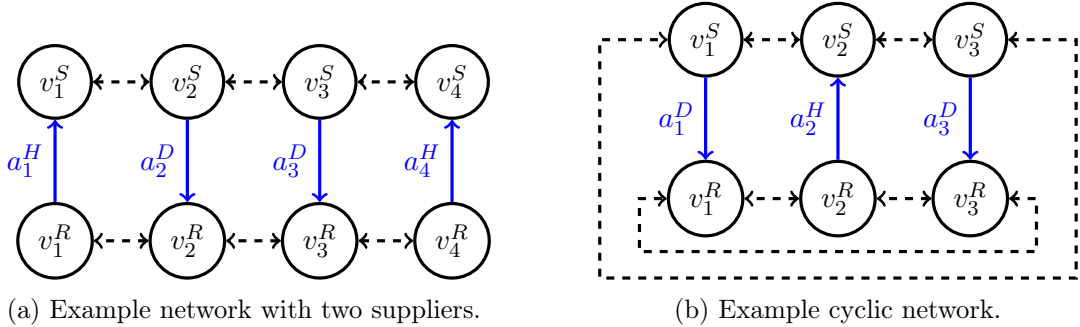

To model physical properties like temperature mixing, we need to identify
the flow direction of the pipe arcs. Thus, we introduce the binary
\emph{flow direction variables} $\mfpos[a]$, $\mfneg[a] \in \{0,1\}$ with
$\mfpos[a] + \mfneg[a] = 1$ for all pipe arcs $a \in \epipe$. The mass flow
rate $\mf[a]$ on a pipe arc $a$ is bounded from above by a maximal mass
flow rate $\mubound[a]$ due to the properties of the underlying pipe
(material, diameter, etc.). The coupling of the flow direction variables
with the mass flow is as follows:
\begin{align*}
  -\mubound[a] \, \mfneg[a] \leq \mf[a] \leq \mubound[a] \, \mfpos[a] \quad
  \forall\, a \in \epipe.
\end{align*}
Maximal mass flow rates of standard pipe types used in district heating
networks can be found in~\cite{Isoplus_starre_Verbundsysteme}.
Note that one can also use the binary variables to model the absolute value
of the flow as $\abs{\mf[a]} = \mf[a] \, \mfpos[a] - \mf[a] \,\mfneg[a]$.

\paragraph{Flow Conservation}
At each node of the network, mass flow has to be conserved, i.e., the water
circulates without loss through the network. These \emph{flow conservation
  constraints} can be expressed as:
\begin{align*}
  \sum_{a \in \delta^{+}(v)} \mf[a] - \sum_{a \in \delta^-(v)} \mf[a] =
  0\quad\forall\, v \in \nodes,
\end{align*}
where $\delta^+(v) \define \{a = (v,w) \in \edges\}$ are the arcs leaving
$v$ and $\delta^-(v) \define \{a = (u,v) \in \edges\}$ are the arcs
entering $v$.

\subsection{Thermal Variables and Constraints}
\label{sec:TemperatureModel}

Based on the stationary thermal energy equation, see, e.g., Roland and
Schmidt~\cite{dhnExpansion} and K\"ocher~\cite{koecher}, and taking heat
losses into account, the temperature change along a pipe $a \in \epipe$ can
be described as
\begin{align}
  \tend[a] &= (\tstart[a] - \tamb) \, \exp{ \left( -\frac{\heatloss[a]}{\heatcap}\frac{\length[a]}{\abs{\mf[a]}} \right)} + \tamb, \label{eq:texp}
\end{align}
assuming $\mf[a] \neq 0$. Here, $\tstart[a]$ and $\tend[a]$ are the
temperatures in \degreeCelsius\ at the start and end of arc $a$ (in
flow direction), respectively, $\tamb$ is the ambient temperature in
\degreeCelsius, $\heatloss[a]$ is the \emph{heat loss coefficient} in
\kwmc, and~$\heatcap$ is the \emph{specific heat capacity} in
\kjkgc. The coefficient $\heatloss[a]$ denotes the rate of heat
exchange between the water and the surroundings through the pipe
wall. The formula states that increasing start temperatures as well as
decreasing mass flows would result in a larger temperature loss, while
lower start temperatures and high mass flows leads to a lower
temperature loss. For a derivation of the formula from the
one-dimensional stationary thermal energy equation, we refer to,
e.g.,~\cite{dhnExpansion}.

\paragraph{Approximation}
Since~\eqref{eq:texp} is hard to handle, we consider an approximation of
this equation using the first order Taylor polynomial
($\exp{\left(-\frac{ 1 }{x}\right)} \approx 1 - \frac{1}{x}$), which yields
\begin{equation}\label{eq:TempApprox1}
  \tend[a] \approx  ( \tstart[a] - \tamb ) \left(1 -\frac{ \heatloss[a]}{\heatcap} \frac{\length[a]}{ \abs{\mf[a]} } \right) + \tamb.
\end{equation}
For $\mf{} \neq 0$, we can reformulate this to
\begin{equation}\label{eq:TempApprox}
  \tend[a]\, \abs{\mf[a]} = \tstart[a]\, \abs{\mf[a]} - \frac{ \heatloss[a]\,\length[a]}{\heatcap} ( \tstart[a] - \tamb ).
\end{equation}

\begin{remark}
  The approximation in~\eqref{eq:TempApprox} is quite accurate for mass flows
  that are not too small and practical choices of the pipe parameters: if
  $\mf[a] \geq 1.5 \kgs$ the temperature difference is below
  \SI{1}{\celsius}, even for thick and long pipes, see~\cite{Reh25}.
\end{remark}

\begin{remark}
  Under the natural assumption that $\tstart[a]$, $\tend[a] \geq \tamb$,
  Equation~\eqref{eq:TempApprox} yields a lower bound
  $\abs{\mf[a]} \geq \heatloss[a]\,\length[a] / \heatcap$,
  because~\eqref{eq:TempApprox1} implies that
  $\smash{1 -\tfrac{\heatloss[a]}{\heatcap} \frac{\length[a]}{\abs{\mf[a]}}} \geq 0$. Moreover, if $\mf[a]$
  approaches this bound, then $\tend[a]$ tends to $\tamb$.
\end{remark}

Note that~\eqref{eq:TempApprox} depends on the flow direction, because the
temperature $\tstart[a]$ at the start of the arc $a = (u,v)$ depends on
it. With the node temperatures $\tnode[u]$ and $\tnode[v]$ as well as flow
direction variables, we express
\[
  \tstart[a] \define \tnode[u] \, \mfpos[a] + \tnode[v] \, \mfneg[a].
\]

\paragraph{Temperature Mixing}
By the conservation of energy and the assumption of perfect
mixing, it follows that the temperature~$\tnode[v]$ at a node $v \in \nodes$ is
determined by mixing the temperatures~$\tend[a]$ at the end of inflowing
arcs~$a$ scaled by the corresponding mass flow rate~$\mf[a]$ as follows:
\begin{align}\label{eq:TemperatureMixing}
  \tnode[v] = \frac{\displaystyle\sum_{a \in \delta^-(v)} \mfpos[a]\, \mf[a]\, \tend[a]
  - \sum_{a \in \delta^+(v)} \mfneg[a]\, \mf[a]\, \tend[a]}{\displaystyle\sum_{a \in
  \delta^-(v)} \mfpos[a]\, \mf[a] - \sum_{a \in \delta^+(v)} \mfneg[a]\, \mf[a]}.
\end{align}
A derivation of this constraint from the conservation of energy can be
found in \cite{koecher}. We implement this equation in the form obtained by
multiplying the denominator to $\tnode[v]$. To shorten the specification of
the models in the following, we abbreviate~\eqref{eq:TemperatureMixing} as
$\text{TM}_v(\mf[],\mfpos[],\mfneg[],\tend[],\tnode[v])$.

\subsection{Modeling of Suppliers}
\label{sec:SupplierModel}

The produced thermal energy $\heatpower[a]$ of heatings (energy
suppliers) depends on the temperature difference $\tnode[u] - \tend[a]$ on
the arc $a = (u,v) \in \esupplier$ and the mass flow rate $\mf[a]$ as well
as the specific heat capacity $\heatcap$:
\begin{align*}
  \heatpower[a] = \heatcap \, \mf[a] \, (\tnode[u] - \tend[a] ),
\end{align*}
see, e.g., Borsche et al.~\cite{borsche}
and~\cite{koecher,dhnExpansion}. The temperature $\tend[a]$ after the
heating $a \in \esupplier$ is bounded by $\tubound{a}$. Similarly,
$\heatpower[a]$ is lower bounded by $\powerbound[a]$. Note that
$\heatpower[a] \leq 0$.

To enable the flow of the water through the network, the pressure has to be
increased at heatings using pumps. We define the electric power
$\pumppower[a]$ of the pump needed to increase the pressure from $\pr[u]$
to $\pr[v]$ analogously to Krug et al.~\cite{nodhn} by:
\begin{align*}
  \pumppower[a] = \frac{1}{\pumpeff[a]\, \rho} \, (\pr[v] - \pr[u]) \, \mf[a],
\end{align*}
where $\pumpeff[a] \in [0,1]$ denotes the efficiency of a pump.

Since only pressure differences occur in the model, the pressure can
be fixed at one node, e.g., the end node $\bar{v}$ of one heating arc
$a \in \esupplier$, to get uniquely defined pressure values:
\begin{align*}
  \pr[\bar{v}] = \pfix.
\end{align*}
By choosing this value sufficiently large, it is ensured that the pressure
does not fall below a certain pressure value, which is needed to prevent
vaporization. We use \SI{10}{\bar} = \SI[retain-unity-mantissa=false]{1e6}{\pascal}.

\subsection{Modeling of Demands}
\label{sec:DemandModell}

The formula of the thermal energy for a demand arc is analogous to the
heatings, but we assume the energy of the demands $\heatdemand[a]$,
$a \in \edemand$, to be given. We do not model the inner processes of the
consumers, but we instead assume that the outgoing temperature $\tend[a]$ at
every consumer is given by a fixed value $\tfix$.  This assumption is
motivated by the fact that energy transmission operators require a specific
temperature when leaving consumers, such that the return network part and
therefore the incoming temperature at the supplier does not exceed a
maximal temperature.  Thus, for all demand arcs $a = (u,v)\in \edemand$ the
following holds:
\begin{align*}
  \heatdemand[a] = \heatcap\, \mf[a]\, (\tnode[u] - \tend[a]) = \heatcap \, \mf[a] \, (\tnode[u] - \tfix).
\end{align*}

Furthermore, we do not model the pressure control within the
consumers. Instead, to ensure that their heat demand can be fulfilled, we
require a lower bound $\plossmin > 0$ on the pressure difference:
\begin{align*}
  \pr[u] - \pr[v] \geq \plossmin.
\end{align*}

\subsection{Resulting Optimization Model}
\label{sec:OptModel}

The two main cost factors when operating a heating network are the
electrical power of the pumps and the heat energy needed at heating arcs.
Therefore, we consider
$\sum_{a \in \esupplier} ( \pumppower[a] - \heatpower[a])$ as the costs
that we want to minimize; recall that $\heatpower[a] \leq 0$.  The
resulting optimization problem is as follows:
\begin{align}
  \min\quad & \sum_{a \in \esupplier} (\pumppower[a] - \heatpower[a]), \label{eq:objective} \\
  \text{s.t.} &   \sum_{a \in \delta^{+}(v)} \mf[a] - \sum_{a \in \delta^-(v)} \mf[a] = 0 && \forall\, v \in \nodes,\label{eq:massconservation}\\
  & \text{TM}_v(\mf[],\mfpos[],\mfneg[],\tend[],\tnode[v]) && \forall v \in \nodes,\label{eq:tmixing}\\
  & \pr[u] - \pr[v] = \zeta_{a} \, \mf[a] \abs{\mf[a]} - \rho g \, (\height[u] - \height[v]) && \forall\, a = (u,v) \in \epipe, \label{eq:ppipe} \\
  & \pr[u] - \pr[v] \geq \plossmin && \forall\, a = (u,v) \in \edemand, \label{eq:prdemand} \\
  & \pumppower[a] = \frac{1}{\pumpeff[a]\, \rho} \, (\pr[v] - \pr[u]) \, \mf[a] && \forall\, a = (u,v) \in \esupplier, \label{eq:prpump} \\
  & \abs{\mf[a]}\, \tend[a] = \abs{\mf[a]}\, \tstart[a] - \tfrac{ \heatloss[a] \, \length[a]}{\heatcap} \,  (\tstart[a] - \tamb) && \forall\, a \in \epipe, \label{eq:tpipe} \\
  & \heatcap \, \mf[a] \, (\tnode[u] - \tfix) = \heatdemand[a] && \forall\, a = (u,v) \in \edemand, \label{eq:tdemand} \\
  & \heatpower[a] = \heatcap \, \mf[a] \, (\tnode[u] - \tend[a])  && \forall\, a = (u,v) \in \esupplier, \label{eq:theating} \\
  & \tstart[a] = \tnode[u] \, \mfpos[a] + \tnode[v] \, \mfneg[a] && \forall\, a = (u,v) \in \epipe, \label{eq:tstart}\\
  & \mfpos[a] + \mfneg[a] = 1 && \forall\, a \in \epipe, \label{eq:binvars} \\
  & \pr[\bar{v}] = \pfix, && \label{eq:pfix} \\
  & -\mubound[a] \, \mfneg[a] \leq \mf[a] \leq \mubound[a] \, \mfpos[a] && \forall\, a \in \epipe, \label{eq:mfbound} \\
  & \tnode[v] \geq \tamb && \forall\, v \in \nodes, \label{eq:tbound1} \\
  & \tend[a] \geq \tamb && \forall\, a \in \epipe, \label{eq:tbound2} \\
  & \tend[a] = \tfix && \forall a \in \edemand, \label{eq:minTout} \\
  & \tend[a] \leq \tubound{a} && \forall\, a \in \esupplier, \label{eq:maxTin} \\
  & \pumppower[a] \geq 0 && \forall\, a \in \esupplier, \label{eq:pos2} \\
  & \powerbound[a] \leq \heatpower[a] \leq 0 && \forall\, a \in \esupplier, \label{eq:neq} \\
  & \mf[a] \geq 0 && \forall\, a \in \edemand \cup \esupplier,\label{eq:mfdem}  \\
  & \mfpos[a],\; \mfneg[a] \in \{0,1\} && \forall\, a \in \epipe. \label{eq:binary}
\end{align}
Recall that~\eqref{eq:pfix} fixes the pressure at particular node~$\bar{v}$ to $\pfix$.
In Appendix~\ref{sec:ImplementedModel}, we present the model as implemented
in our solver.

\begin{remark}
  One challenge in solving~\eqref{eq:objective}--~\eqref{eq:binary} lies in
  the fact that the demand can be satisfied by a combination (product) of
  temperature and flow rate and similarly for the heat supply. Moreover,
  the nonlinear temperature mixing constraints~\eqref{eq:tmixing} are based
  on the flow directions.
\end{remark}

\paragraph{Degrees of Freedom}
The variables for $v \in \nodes$ are $\tnode[v]$ and $\pr[v]$. The
variables for $a \in \epipe$ are $\mf[a]$, $\mfpos[a]$, $\mfneg[a]$,
$\tstart[a]$, $\tend[a]$. For $a \in \esupplier$ they are
$\mf[a]$, $\pumppower[a]$, $\tend[a]$, $\heatpower[a]$. Finally, for $a \in
\edemand$ they are $\mf[a]$ and $\tend[a]$. This yields $2\, \card{\nodes} + 5\,
\card{\epipe} + 4\, \card{\esupplier} + 2\, \card{\edemand}$ variables.

On every node we have a constraint for the mass
conservation~\eqref{eq:massconservation} and temperature
mixing~\eqref{eq:tmixing}.  Recall that the mass conservation constraints
have rank $\card{V} - 1$. Thus, we have $2\, \card{\nodes} -1 $ linearly
independent mass conservation and temperature mixing constraints.

Additionally, for $a \in \epipe$ , we have Equations~\eqref{eq:ppipe},
\eqref{eq:tpipe}, \eqref{eq:tstart}, \eqref{eq:binvars}, and
\eqref{eq:mfbound}.  For $a \in \esupplier$, we have
Equations~\eqref{eq:prpump} and \eqref{eq:theating}. For $a \in \edemand$,
we have Equations~\eqref{eq:minTout} and~\eqref{eq:tdemand}. Together with
Equation~\eqref{eq:pfix}, we get
$2\, \card{V} + 5\, \card{\epipe} + 2\, \card{\esupplier} + 2\,
\card{\edemand}$ equations. Thus, there are $2\, \card{\esupplier}$ more
variables than equations.

\section{Properties of Stationary Heating Networks}
\label{sec:Properties}

In this section we derive useful properties of solutions of heating networks.

\subsection{Flow Directions and Acyclicity}
\label{sec:FlowDirection}

In flexible heating networks, the flow directions are important. If they
are fixed, an LP-relaxation of Constraints~\eqref{eq:tmixing}
and~\eqref{eq:ppipe} can be made much tighter, which generally improves the
solution speed.

Following the ideas in~\cite{HabP24}, the flow direction variables
$\mfpos[a]$ and $\mfneg[a]$ can be used to derive restrictions that
possibly allow to fix directions during presolving. For instance, because
of flow conservation, the flow of at least one arc must enter a node and
the flow of at least one arc must leave it:
\begin{subequations}
  \label{eq:BinaryFlowConservation}
  \begin{align}
    \sum_{a \in \delta^-(v)} \mfpos[a] + \sum_{a \in \delta^+(v)} \mfneg[a] \geq 1 & \quad \forall v \in \nodes, \\
    \sum_{a \in \delta^+(v)} \mfneg[a] + \sum_{a \in \delta^-(v)} \mfpos[a] \geq 1 & \quad \forall v \in \nodes.
  \end{align}
\end{subequations}
Since for arcs $a \in \esupplier \cup \edemand$, $\mfpos[a]$ can be fixed
to 1 and $\mfneg[a]$ can be fixed to 0, the inequalities might only contain
one unfixed variable, which can then instead directly be fixed as well, or
the inequality is redundant.

\begin{figure}[tb]
  \centering
  \begin{minipage}[t]{.49\linewidth}
    \centering
    \begin{tikzpicture}[scale=0.7]
      \node at (0,3) [circle,draw, very thick] (va){$v^S_1$};
      \node at (2.5,3) [circle,draw, very thick] (vb){$v^S_2$};
      \node at (5,3) [circle,draw, very thick] (vc){$v^S_3$};
      \node at (7.5,3) [circle,draw, very thick] (vd){$v^S_4$};
      \node at (0,0) [circle,draw, very thick] (ra){$v^R_1$};
      \node at (2.5,0) [circle,draw, very thick] (rb){$v^R_2$};
      \node at (5,0) [circle,draw, very thick] (rc){$v^R_3$};
      \node at (7.5,0) [circle,draw, very thick] (rd){$v^R_4$};
      \draw[->, color=blue, very thick] (ra)-- node[left]{$a^{H}_{1}$}(va);
      \draw[->, color=blue, very thick] (rd)-- node[left]{$a^{H}_{4}$}(vd);
      \draw[->, color=blue, very thick] (va) -- (vb);
      \draw[<->, very thick, dashed] (vb) -- (vc);
      \draw[->, color=blue, very thick] (vc)-- node[left]{$a^{D}_{3}$}(rc);
      \draw[<->, very thick, dashed] (rc)-- (rb);
      \draw[->, color=blue, very thick] (rb) -- (ra);
      \draw[->, color=blue, very thick] (vb)-- node[left]{$a^{D}_{2}$}(rb);
      \draw[->, color=blue, very thick] (vd) -- (vc);
      \draw[->, color=blue, very thick] (rc) -- (rd);
    \end{tikzpicture}
    \subcaption{Example network with two suppliers.}
    \label{fig:2_bid_regina}
  \end{minipage}
  \hfill%
  \begin{minipage}[t]{.49\linewidth}
    \centering
    \begin{tikzpicture}[scale=0.7]
      \node at (0,3) [circle,draw, very thick] (va){$v^S_1$};
      \node at (2.5,3) [circle,draw, very thick] (vb){$v^S_2$};
      \node at (5,3) [circle,draw, very thick] (vc){$v^S_3$};
      \node at (0,0) [circle,draw, very thick] (ra){$v^R_1$};
      \node at (2.5,0) [circle,draw, very thick] (rb){$v^R_2$};
      \node at (5,0) [circle,draw, very thick] (rc){$v^R_3$};
      \draw[->, color=blue, very thick] (va)-- node[left]{$a^{D}_{1}$}(ra);
      \draw[->, color=blue, very thick] (rb)-- node[left]{$a^{H}_{2}$}(vb);
      \draw[->, color=blue, very thick] (vc)-- node[left]{$a^{D}_{3}$}(rc);
      \draw[<-, color=blue, very thick] (va) -- (vb);
      \draw[->, color=blue, very thick] (vb) -- (vc);
      \draw[<-, color=blue, very thick] (rb) -- (ra);
      \draw[->, color=blue, very thick] (rc) -- (rb);
      \draw[<->, dashed, very thick] (rc) -- (6.25,0) -- (6.25,-1) -- (-1.25,-1) -- (-1.25,0) -- (ra);
      \draw[<->, dashed, very thick] (vc) -- (7,3) -- (7,-1.5) -- (-2,-1.5) -- (-2,3) -- (va);
    \end{tikzpicture}
    \subcaption{Example network with cycle.}
    \label{fig:2_bid_circulus}
  \end{minipage}
  \caption{Example networks with propagated mass flow directions (dashed
    black arcs). Supply nodes are on top on the left and outside on the right.}
  \label{fig:2_example_bid}
\end{figure}
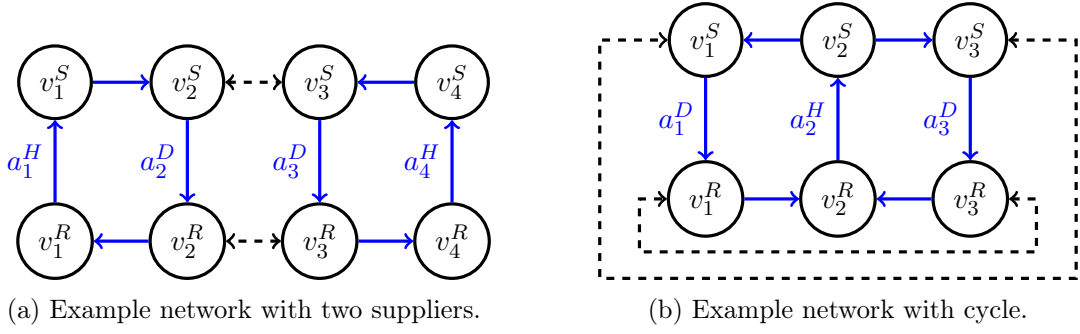

Moreover, there cannot be nonzero flow along cycles in the supply or return part alone.
To this end, let~$C$ be a cycle in $\esup$ or $\eret$, i.e., $C$ forms a
cycle in the underlying undirected graph. Fix an arbitrary direction of $C$
and let $C^+$ and $C^-$ be the subsets of arcs in $C$ in forward and backward
direction, respectively. There is a \emph{nonzero flow along $C$}, if
$\mf[a]$ is positive for arcs in $C^+$ and negative for $C^-$, or
conversely.

\begin{lemma}\label{lem:NoCycle}
  Let $C \subseteq \esup$ or $C \subseteq \eret$ be a cycle. Then in every
  feasible solution, there cannot exist a nonzero flow along $C$.
\end{lemma}

\begin{proof}
  Summing the pressure differences on the left hand side of
  Constraints~\eqref{eq:ppipe} along $C^+$ and their negatives along $C^-$ yields
  \[
    \sum_{a = (u,v) \in C^+} (\pr[u] - \pr[v]) - \sum_{a = (u,v) \in C^-} (\pr[u] - \pr[v]) = 0,
  \]
  because the pressures cancel out. The sum of the right hand sides yields:
  \begin{align*}
    & \sum_{a = (u,v) \in C^+} \big(\zeta_{a} \, \mf[a] \abs{\mf[a]} - \rho g \, (\height[u] - \height[v])\big)
    - \sum_{a = (u,v) \in C^-} \big(\zeta_{a} \, \mf[a] \abs{\mf[a]} - \rho g \, (\height[u] - \height[v])\big)\\
    = & \sum_{a = (u,v) \in C^+} \zeta_{a} \, \mf[a] \abs{\mf[a]} - \sum_{a = (u,v) \in C^-} \zeta_{a} \, \mf[a] \abs{\mf[a]}
  \end{align*}
  because the heights cancel out along the cycle as well. A nonzero flow
  along $C$ would imply that this sum is nonzero, which is a contradiction.
\end{proof}

This implies the validity of the following constraints:
\begin{subequations}\label{eq:CycleConstraints}
  \begin{align}
    \sum_{a \in C^+} \mfpos[a] + \sum_{a \in C^-} \mfneg[a] \leq \card{C} - 1,\\
    \sum_{a \in C^+} \mfneg[a] + \sum_{a \in C^-} \mfpos[a] \leq \card{C} - 1.
  \end{align}
\end{subequations}

\subsection{Mass Flow Equality in Supply and Return Part}
\label{sec:FlowEquality}

In the following, we will prove that corresponding flows on the supply and
return part are equal. To this end, we generally assume that each node
$v^S$ in the supply part has a unique \emph{corresponding} node $v^R$ in
the return part, and similarly each arc $a^S$ in the supply part has a
unique corresponding arc $a^R$ in the return part with the same pipe
parameters. This implies that the topology of the supply and return part
are identical.

We first show that the flow along a node in the supply part is the negative
of the flow along the corresponding node in the return part.

\begin{lemma}\label{lem:NodeFlowSupRet}
  Let $\graph = (\nodes, \edges)$ be a heating network. Consider a node
  $v^S \in \nsup$ in the supply part with corresponding node
  $v^R \in \nret$ in the return part. Then
  \begin{align*}
    \sum_{a \in \delta^{+}_{S}(v^S)} \mf[a] - \sum_{a \in \delta^-_{S}(v^S)} \mf[a] =
    - \Big(\sum_{a \in \delta^{+}_{R}(v^R)} \mf[a] - \sum_{a \in \delta^-_{R}(v^R)} \mf[a]\Big),
  \end{align*}
  where $\delta^{+}_{S}(v^S) \define \delta^+(v^S) \cap \esup$ and
  similarly for $\delta^-_S$, $\delta^+_R$, and $\delta^-_R$.
\end{lemma}

\begin{proof}
  Recall that by our general assumption, the nodes $v^S$ and $v^R$ are
  connected by either a demand or a heating arc. Flow conservation implies
  that
  \[
    \sum_{a \in \delta^{+}(v^S)} \mf[a] - \sum_{a \in \delta^-(v^S)} \mf[a] \pm \mf[b] = 0,
  \]
  where $b \in \edemand \cup \esupplier$ and $\mf[b]$ is the flow over this
  arc; its sign depends on the type of arc, i.e., whether
  $b \in \edemand$ or $b \in \esupplier$. Similarly,
  \[
    \sum_{a \in \delta^{+}(v^R)} \mf[a] - \sum_{a \in \delta^-(v^R)} \mf[a] \mp \mf[b] = 0,
  \]
  Solving for $\mf[b]$ and equating yields the claim.
\end{proof}

We next prove the main result of this section.

\begin{theorem}\label{thm:SupRetEqual}
  Let $\graph = (\nodes, \edges)$ be a heating network with supply graph
  $\gsup = (\nsup, \esup)$ and return graph $\gret = (\nret, \eret)$. For a
  supply arc $a^S \in \esup$, let $a^R$ be the corresponding return arc,
  where we assume that they are oriented in opposite directions and that
  $\zeta_{a^S} = \zeta_{a^R}$. Then for every feasible solution
  $\mf[a^S] = \mf[a^R]$.
\end{theorem}

\begin{proof}
  Assume that there exists a supply arc $a_1^S = (v_1^S, v_2^S) \in \esup$
  with corresponding return arc $a_1^R = (v_2^R, v_1^R) \in \eret$ such
  that $\smash{\mf[a_1^S] \neq \mf[a_1^R]}$ (note that we can assume w.lo.g.\ that
  $a_1^S$ and $a_1^R$ are oriented in opposite directions). By possibly
  reorienting the two arcs (and thus multiplying the flows by $-1$), we can
  assume w.l.o.g.\ that $\mf[a_1^S] > \mf[a_1^R]$.

  If $v_2^S$ is a leaf node of $\gsup$, i.e., it has degree 1 in this
  graph, then flow conservation implies a contradiction to the
  assumption. Otherwise, by Lemma~\ref{lem:NodeFlowSupRet}, we have
  \begin{align*}
    \sum_{a \in \delta^{+}_{S}(v_2^S)} \mf[a] - \sum_{a \in \delta^-_{S}(v_2^S)} \mf[a]
    + \sum_{a \in \delta^{+}_{R}(v_2^R)} \mf[a] - \sum_{a \in \delta^-_{R}(v_2^R)} \mf[a] = 0.
  \end{align*}
  This implies
  \begin{align*}
    \mf[a_1^S] - \mf[a_1^R] = \sum_{a \in \delta^{+}_{S}(v_2^S)} \mf[a] - \sum_{a \in \delta^-_{S}(v_2^S) \setminus \{a_1^S\}} \mf[a]
    + \sum_{a \in \delta^{+}_{R}(v_2^R) \setminus \{a_1^R\}} \mf[a] - \sum_{a \in \delta^-_{R}(v_2^R)} \mf[a].
  \end{align*}
  The assumption $\mf[a_1^S] > \mf[a_1^R]$ implies that
  \begin{align*}
    \sum_{a \in \delta^{+}_{S}(v_2^S)} \mf[a] - \sum_{a \in \delta^-_{S}(v_2^S) \setminus \{a_1^S\}} \mf[a]
    > \sum_{a \in \delta^-_{R}(v_2^R)} \mf[a] - \sum_{a \in \delta^{+}_{R}(v_2^R) \setminus \{a_1^R\}} \mf[a],
  \end{align*}
  which yields that
  \begin{align*}
    \sum_{a \in \delta^{+}_{S}(v_2^S)} \mf[a] > \sum_{a \in \delta^-_{R}(v_2^R)} \mf[a]
    \quad\text{ or }\quad
    \sum_{a \in \delta^-_{S}(v_2^S) \setminus \{a_1^S\}} \mf[a] < \sum_{a \in \delta^{+}_{R}(v_2^R) \setminus \{a_1^R\}} \mf[a].
  \end{align*}
  In the first case, there exists a supply arc $a_2^S = (v_2^S, v_3^S)$ and
  corresponding return arc $a_2^R = (v_3^R, v_2^R)$ such that
  $\mf[a_2^S] > \mf[a_2^R]$. In the second case, there exists
  $a_2^S \in \delta^-_S(v_2^S) \setminus \{a_1^S\}$ and
  $a_2^R \in \delta^{+}_{R}(v_2^R) \setminus \{a_1^R\}$ with
  $\mf[a_2^S] < \mf[a_2^R]$; we then reverse the orientation of $a_2^S$ and
  $a_2^R$ (and multiply the flows by $-1$) to get $a_2^S = (v_2^S, v_3^S)$
  and $a_2^R = (v_3^S, v_2^S)$. Thus, in both cases,
  $\mf[a_2^S] > \mf[a_2^R]$.

  We can then repeat this process, producing one new arc pair
  $(a_k^S, a_k^R)$ with $\smash{\mf[a_k^S] > \mf[a_k^R]}$ in each iteration $k$
  until $v_k^S$ is a leaf node in $\gsup$ or we reach a node that we have
  already considered. In the first case, we get a contradiction as
  above. Thus, assume that we repeat nodes, i.e., there exists a cycle
  $C^S \subseteq \esup$ with corresponding return part
  $C^R \subseteq \eret$ such that $\mf[a^S] > \mf[a^R]$ for all
  $a^S \in C^S$ with corresponding $a^R$ in $C^R$. By
  Lemma~\ref{lem:NoCycle}, there cannot be a nonzero flow along $C^S$ (and
  $C^R$), which contradicts our assumption.
\end{proof}

Similarly, the pressure differences in the return part are the negatives of
the pressure differences in the supply part.

\section{Numerical Results for Test Networks}
\label{sec:NumericalResultsI}

In this section, we present numerical results for several test networks.
We use two types of instances, generated networks
(see Section~\ref{sec:TestNetworks}) and reduced real-world networks
(Section~\ref{sec:ResultsReal-World}). The goal is to illustrate the
  effects of the different methods that we apply.

\subsection{Generation of Test Networks}
\label{sec:TestNetworks}

In this section, we very briefly describe the generation of test networks
with different sizes and properties. For more details, we refer
to~\cite{Reh25}.

The networks are generated with a given number $N$ of nodes in the supply
network, a number $\card{\esupplier}$ of suppliers, and a number $\card{C}$ of
cycles in the network. Based on this input data, randomized coordinates for
the supply nodes $\nsup$ are chosen and the nodes are connected via pipe
arcs ($\esup$), first such that a tree is formed and then possibly adding
arcs to create cycles. The coordinates are generated such that the length
of the connecting supply arcs is between \SI{1}{\meter} and
\SI{200}{\meter}. Then the corresponding return network is generated. The
coordinate of a return node equals the coordinate of the corresponding
supply node, but the $y$-coordinate is shifted by 1. The positions of the
supplier arcs $\esupplier$ and the demand arcs $\edemand$ are determined
such that every supply node is connected to the corresponding return node
via a supplier or a demand arc. The arrangement of supplier and demand arcs
is randomized. Table~\ref{tab:TestNetworks} contains a list of the
characteristics of the resulting $26$ test networks that we use in the
following.

\begin{table}[tb]
  \caption{Test networks in the gen-testset}
  \centering
  \begin{tabular}{@{}rrrr@{}}
    \toprule
    \# & $N$ & $\card{\esupplier}$ & $\card{C}$ \\
    \midrule
    1 & 10 & 1 & 0 \\
    2 & 50 & 1 & 0 \\
    3 & 70 & 1 & 0 \\
    4 & 10 & 2 & 0 \\
    5 & 50 & 2 & 0 \\
    6 & 50 & 10 & 0 \\
    7 & 70 & 10 & 0 \\
    8 & 100 & 10 & 0 \\
    9 & 10 & 1 & 1 \\
    10 & 50 & 1 & 1 \\
    11 & 70 & 1 & 1 \\
    12 & 10 & 2 & 1 \\
    13 & 50 & 2 & 1 \\
    \bottomrule
  \end{tabular}
  \qquad
  \begin{tabular}{@{}rrrr@{}}
    \toprule
    \# & $N$ & $\card{\esupplier}$ & $\card{C}$ \\
    \midrule
    14 & 10 & 2 & 2 \\
    15 & 50 & 2 & 2 \\
    16 & 70 & 2 & 2 \\
    17 & 100 & 2 & 2 \\
    18 & 70 & 5 & 2 \\
    19 & 120 & 10 & 2 \\
    20 & 150 & 20 & 2 \\
    21 & 150 & 30 & 2 \\
    22 & 50 & 10 & 4 \\
    23 & 70 & 10 & 4 \\
    24 & 120 & 10 & 4 \\
    25 & 150 & 10 & 4 \\
    26 & 150 & 20 & 4 \\
    \bottomrule
  \end{tabular}

  \label{tab:TestNetworks}
\end{table}

The pipes are made of pre-insulated rigid steel with specifications
as found in the manual for industrial
pipes~\cite{Isoplus_starre_Verbundsysteme}. Since we only consider primary
distribution lines, we use pipes of the sizes DN200 up to DN450;
see~\cite{Reh25} for more information.

Common supply temperatures are between \SIlist{70;150}{\celsius}, see
also~\cite{Planungshandbuch}. They depend on the size of the network, the
technical conditions of the producers, and the requirements of the connected
consumers. While consumers with a radiator require temperatures from
\SIrange{70}{90}{\celsius} (see~\cite{Planungshandbuch}), underfloor
heating systems can operate at much lower temperatures, see Lund et
al.~\cite{LUND20141}. Thus, the overall required temperatures at the consumers
range from \SIrange{60}{90}{\celsius}. Heat suppliers can be, for example,
combined heat and power plants (CHP), geothermal, solar thermal energy, or
waste heat possibly combined with heat pumps. For generating high
temperatures, combined heat and power plants are a favorable supplier
type~\cite{Planungshandbuch, geothermal}. In this paper, we do not model
the inner processes of different supplier types, but use exemplary networks
with different supplier types. Thus, we assume to have at least one large
supplier, which ensures the feasibility of the demands. This main
supplier has a maximal temperature bound \tubound{a} of \SI{150}{\celsius}.  The
other suppliers have maximal temperatures from
\SIrange{80}{150}{\celsius}.  To imitate the different supplier types and
dimensions, we also use different heat capacities~$\powerbound[a]$.

We scale the maximal heat demand according to the number of demand arcs. It can
range from $10$ up to $100$ times the number of demand arcs. For a network
with~10 demand arcs, for example, the maximal power can range from
\SIrange{100}{1000}{\kilo\watt}, while in a network with~100 demand arcs,
the maximal power can range from \SIrange{1000}{10000}{\kilo\watt}.  Since
the heat demand in a network varies due to seasons and daily heat load
variations, e.g., peaks in households in the morning or in the afternoon
and minimal heat demand during the night time, we consider different load
factors for our determined networks. For each network, we consider four
different scenarios that represent load factors between 30\,\% and 95\,\%
utilization of the network, see \cite{Planungshandbuch}.

\subsection{Implementation}\label{sec:Implementation}

We have implemented a solution algorithm for computing globally optimal
operations of district heating networks in C using SCIP~\cite{SCIP10}.
SCIP provides a framework for solving
mixed-integer nonlinear programs (MINLP). The core algorithm is spatial
branch-and-bound~\cite{TawS04} using bounds based on linear programming
(LP). The relaxations are computed using McCormick inequalities for the
product of two variables and gradient inequalities for convex
expressions. Spatial branching splits the variable domains into smaller
pieces allowing to refine the relaxations. This is recursively iterated
until a given feasibility accuracy is obtained. Our code reads problem
data in JSON format and builds the model. Then it calls SCIP to solve
it. Additionally, we have implemented the following methods:
\begin{itemize}
\item \texttt{Cycle}: Constraints~\eqref{eq:BinaryFlowConservation}
  are added at the beginning. Moreover,
  Constraints~\eqref{eq:CycleConstraints} are generated by running a
  depth-first search (DFS) over the supply/return part and
  adding~\eqref{eq:CycleConstraints} if a cycle is found. Note that
  this will only consider fundamental cycles, but for most real-world
  networks this will actually produce all cycles.
\item \texttt{SupRet}: We use Theorem~\ref{thm:SupRetEqual} to make
  the flows in the supply and return part equal.
\item \texttt{MixSimpl}: When applying standard presolving techniques
  as implemented in SCIP, in many cases the flow directions are
  already fixed. This can be used to simplify the temperature mixing
  constraints~\eqref{eq:tmixing}. In particular, if there is only one
  arc that carries flow into a node, no mixing is necessary. We
  therefore first run standard presolving and then test for these
  cases. If they appear, we replace~\eqref{eq:tmixing} by linear
  constraints. We then again presolve the problem if changes were
  made.
\item \texttt{LPHeur}: The computation of primal solutions of our optimization
  problem~\eqref{eq:objective}--\eqref{eq:binary} is one of the challenges
  for solving the presented model. We therefore implemented two \emph{primal
  heuristics}. The first heuristic runs every 10th level of the
  branch-and-bound tree. A copy of the current subproblem is created.
  Based on the current LP-solution, the binary flow direction variables are
  fixed, which results in a nonlinear program (NLP) containing only
  continuous variables. Then the resulting subproblem is solved with a node
  limit, which also calls an NLP-solver with suitable initial points. If
  the flow directions are correctly guessed and the initial point is good
  enough, this will very often produce good primal solutions.
\item \texttt{HydrHeur}: The second primal heuristic first computes a solution to the
  ``hydraulic'' problem that arises by ignoring all temperature variables
  and constraints. However, bounds on the flows from~\eqref{eq:tdemand} are
  taken into account. The resulting problem is solved with a limit on the
  number of nodes in the brand-and-bound tree. If a solution is found, it
  is used to fix the flow direction variables. Then the corresponding
  subproblem is solved as for \texttt{LPHeur}.
\item \texttt{BranchTree}: Since we assume positive pump costs, there
  must exist one demand arc in which the pressure difference between
  supply and return part is equal to the bound $\plossmin$. We
  implemented two \emph{branching rules}. Both are only applied at the
  root node and generate one child node for each possibility, in which
  we add the constraint $p_u - p_v = \plossmin$.

  In the first rule, we only branch on demand arcs that correspond to
  leaf nodes, if the underlying supply and return graphs $G^S$ and
  $G^R$, respectively, form a tree.
\item \texttt{BranchAll}: The second rule is run for arbitrary
  graphs. A branching node is generated for each demand arc (not only
  the leaf nodes).
\end{itemize}

\subsection{Experimental Setup}
\label{sec:ExperimentalSetup}

\begin{table}
  \caption{Comparison of different algorithm variants.}
  \label{tab:All}
  \begin{tabular*}{\textwidth}{@{}l@{\;\;\extracolsep{\fill}}rrr@{}}\toprule
    Setting & \#opt & time [s] & \#nodes \\
    \midrule
    Base                 & \num{  20} & \num{ 2342.01} & \num{ 35509.4} \\
    Cycle                & \num{  29} & \num{ 1721.58} & \num{ 25906.7} \\
    SupRet               & \num{  22} & \num{ 2121.71} & \num{ 35785.8} \\
    MixSimpl             & \num{  27} & \num{ 1954.06} & \num{ 30779.8} \\
    HydrHeur             & \num{  28} & \num{ 1778.78} & \num{ 23782.5} \\
    LPHeur               & \num{  27} & \num{ 2442.94} & \num{ 10381.6} \\
    Combined             & \num{  50} & \num{  958.86} & \num{  3702.6} \\
    BranchTree           & \num{  50} & \num{  925.32} & \num{  3567.0} \\
    BranchAll            & \num{  47} & \num{ 1032.61} & \num{  4549.3} \\
  \bottomrule
  \end{tabular*}
\end{table}

The code and generated instances are available at \url{https://github.com/dopt-TUDa/heatnet}.
We use developer version 10.0.2 of SCIP~\cite{SCIP10} with CPLEX
(12.10.0.0) as LP solver and IPOPT~\cite{IPOPT} (3.14.20) as NLP
solver. The experiments were run on a Linux cluster with 3.5 GHz Intel Xeon
E5-1620 Quad-Core CPUs, having 32~GB main memory. All computations were run
single-threaded.

We test the following settings, based on the methods and names
introduced in Section~\ref{sec:Implementation}. We use one setting for
applying only \texttt{Cycle}, \texttt{SupRet}, \texttt{MixSimpl},
\texttt{HydrHeur}, and \texttt{LPHeur} each. In addition, the
following combinations are tested:
\begin{itemize}[itemsep=0ex,parsep=0ex,label=$\circ$]
\item \texttt{Base}: None of the new methods is applied.
\item \texttt{Combined}: Apply all \texttt{Cycle}, \texttt{SupRet},
  \texttt{MixSimpl}, \texttt{HydrHeur}, and \texttt{LPHeur} methods.
\item \texttt{BranchTree}: \texttt{Combined} and additionally \texttt{BranchTree}.
\item \texttt{BranchAll}:  \texttt{Combined} and additionally \texttt{BranchAll}.
\end{itemize}

\subsection{Results for Generated Instances}

The results for a time limit of two hours for the generated instances are
given in Table~\ref{tab:All}. The columns provide the name of the variant,
the number of instances solved to optimality, the overall running time in
seconds and number of nodes of the branch-and-bound-tree in shifted
geometric mean\footnote{The \emph{shifted geometric mean} of values
  $t_1, \dots, t_n$ is $\big(\prod(t_i + s)\big)^{1/n} - s$ with shift
  $s = 1$ for time and $s = 100$ for branch-and-bound nodes.}.

The results show that all five individual methods improve upon
\texttt{Base} in the number of solved instances. All these methods
improve the solution time, except \texttt{LPHeur}, because for the easy
instances this heuristic uses too much time. Among these five
variants, \texttt{Cycle} has the largest impact, followed by
\texttt{HydrHeur}.

When applying all five methods in \texttt{Combined}, the impact is
even larger: The number of solved instances increases to 50 and thus
more than doubles this number in comparison to \texttt{Base}. The
running time decreases by about a factor of 0.4. Variant
\texttt{BranchTree} also solves the same number of instances as
\texttt{Combined}, but is \SI{3}{\percent} faster. However, variant
\texttt{BranchAll} solves less instances and is slower than
\texttt{Combined}.

One conclusion of these results seems to be that variants
\texttt{BranchTree} and \texttt{Combined} are the methods of choice,
depending on how many instances are tree-shaped.

In addition, we note that on average, equality of the flows of the supply
and return part is enforced for 72.1 arcs (if turned on). On average 6.3
cycle constraints are generated and 86.6 temperature mixing constraints can
be simplified after presolving of the flow directions (if turned on).

\subsection{Results for Real-World Networks}
\label{sec:ResultsReal-World}

\begin{table}[tb]
  \caption{Statistics on real-world networks.}
  \label{tab:RealStat}
  \centering
  \begin{tabular}{@{}lrrrrr@{}} 
    \toprule
    Network & $\card{\nodes}$ & $\card{\esupplier}$ & $\card{\edemand}$ & $\card{\epipe}$ & $\heatdemand[]$ $[kW]$ \\
    \midrule
    DaNo2020 & 85 & 2 & 13 & 81 & 6516.89 \\  
    DaNo2030 & 45 & 6 & 6 & 40 & 12009.30 \\
    TU2020   & 77 & 1 & 18 & 76 & 15248.22 \\
    TU2030   & 64 & 9 & 9 & 62 & 22226.64 \\
    \bottomrule
  \end{tabular}
\end{table}

\begin{table}
  \caption{Comparison of different algorithm variants.}
  \label{tab:Real}
  \begin{tabular*}{\textwidth}{@{}l@{\;\;\extracolsep{\fill}}rrrrr@{}}\toprule
    Instance & time [s] & \#nodes & dual bnd & primal bnd & gap \% \\
    \midrule
    \multicolumn{6}{@{}l@{}}{\texttt{Combined}}\\
    DaNo2020                          & \num{ 2955.59} & \num{   47161} & \num{  798600} & \num{  806075} &    0.936 \\
    DaNo2030                          & \num{18000.01} & \num{  170010} & \num{  470061} & \num{  610989} &   29.981 \\
    TU2020                            & \num{  293.85} & \num{    3246} & \num{  872947} & \num{  881668} &    0.999 \\
    TU2030                            & \num{18000.00} & \num{    2798} & \num{ 22226.7} & \num{  792400} & 3465.090 \\
    \addlinespace
    \multicolumn{6}{@{}l@{}}{\texttt{BranchAll}}\\
    DaNo2020                          & \num{ 3151.64} & \num{   51999} & \num{  798104} & \num{  806075} &    0.999 \\
    DaNo2030                          & \num{18000.01} & \num{  271390} & \num{  565362} & \num{  610989} &    8.070 \\
    TU2020                            & \num{  390.10} & \num{    4741} & \num{  873096} & \num{  881668} &    0.982 \\
    TU2030                            & \num{18000.00} & \num{   44110} & \num{  156420} & \num{  784883} &  401.778 \\
    \addlinespace
    \bottomrule
  \end{tabular*}
\end{table}

We next present results for real-world networks from the city of
Darmstadt. The underlying network data were provided by the local
energy supplier ENTEGA Plus GmbH and TU Darmstadt. The pipes of the
original data were contracted as described
in~\cite{BOTT21}. Additionally, the heat demands within districts are
aggregated. The original consumers are grouped by districts such that
each consumer in the resulting model represents a district.

We consider two settings. First, we assume the ``today'' scenario, a
centralized network structure with a combined heat and power plant
together with a gas boiler. Second, a ``future'' scenario, with a more
decentralized network structure. We assume that the existing network
expands, since more consumers are connected to the network, i.e., the
demanded heat increases, and additionally the number of heat suppliers
increases, due to, for example, the usage of industrial waste
heat. Note that for better tractability, the future network
structures are less precise than the today's networks, i.e., the
number of nodes, edges and consumers is reduced further.

Table~\ref{tab:RealStat} provides statistics for these networks, like the
number of nodes $\card{\nodes}$, suppliers $\card{\esupplier}$, demands
$\card{\edemand}$, and pipes $\card{\epipe}$. Additionally, the sum of all
demanded heat $\heatdemand[] \define \sum_{a \in \edemand} \heatdemand[a]$
is shown. The \emph{today} networks are denoted by \emph{2020} and the
\emph{future} networks by \emph{2030}.

The results of the variants \texttt{Combined} and \texttt{BranchAll}
are shown in Table~\ref{tab:Real} for a time limit of
\SI{18000}{\second} (\texttt{BranchTree} is not applicable, because
all networks contain cycles). The columns provide the time and number
of branch-and-bound-nodes in shifted geometric mean, the final dual
bound ($d$) and primal bound ($p$), and the gap $(p - d)/d$, in
percent. We use a gap limit of \num{0.01} (\SI{1}{\percent}), i.e., we
stop if the gap is below this value.

The results show that both variants solve instances \texttt{DaNo2020} and
\texttt{TU2020} (within the gap limit). Note that \texttt{BranchAll}
produces a significantly smaller gap for \texttt{DaNo2030} and
\texttt{TU2030}, but is slower for the other two (solved) instances. Both
variants compute the same primal solution bounds, except that
\texttt{BranchAll} improves the primal bound for \texttt{TU2030}. The
results show that depending on the complexity of the networks, the
corresponding optimization problems can be solved to global optimality
within reasonable time and quality.

\section{Heating Networks with Storages}
\label{sec:Storages}

In this section, we consider the presence of heat storages in a
quasi-stationary model, i.e., there are time steps
$T = \{1, \dots, \card{T}\}$, each with a stationary
behavior. These~$\card{T}$ problems are coupled through the fill
status of thermal storages. We will analyze conditions under which a
decomposition approach of the time steps allows for computing an
optimal solution.

\subsection{Modeling of Heat Storages}\label{sec:storages}

In this paper, we use a very simple storage model. We refer, e.g., to
\cite{Planungshandbuch} for a detailed description of the properties of
thermal energy storages and different technologies.

Thermal energy storages can charge heat and store it in a tank for a
period of time and release it when needed.  In general, different storage
technologies exist, but in district heating networks, mainly \emph{sensible
  energy storages} are used~\cite{Planungshandbuch}, in which no phase
change occurs, i.e., storage tanks with water must be under pressure if the
temperature is above \SI{100}{\celsius} to avoid vaporization. In this
paper, we consider short-term storages, since we deal with the
operation of heating networks over one day.

The simplest design for heat storages in district heating networks is the
so-called \emph{direct loading}, which we will also use. This means that in
contrast to indirect loading via heat exchangers, there is no hydraulic
separation of the storage and the connected pipes.

Analogously to the heating and demand arcs, the heat power of a storage
depends on the temperature difference, the specific heat capacity, and the
mass flow. Charging the storage corresponds to heating the tank, while
discharging corresponds to cooling the tank.

In general, the amount of stored energy is not equal to the amount of
energy that can be discharged, due to temperature losses over time. These
losses depend on the thermal insulation of the storage tank and
the difference to the ambient temperature; for more details we
refer to \cite{Planungshandbuch}.  Additionally, a storage has a capacity
for the amount of heat power that can be stored and a charging and
discharging capacity, which describes how quickly a storage can charge or
discharge a certain amount of energy. Again, for more details we refer to
\cite{Planungshandbuch}.

Since we are not interested in storage design or dimensioning, we
consider a simplified storage model in our optimization problem, in which
inner processes of a storage are not represented. In particular, we neglect the
charging and discharging capacity of the storage as well as cooling over
time. Thus, we assume a `perfect' storage, for which the amount of energy
that was charged equals the amount of energy that can be discharged.

\section{Time Dependent Optimization Model with Storages}
\label{sec:StorageOpt}

For the time dependent optimization problem, all variables from
Section~\ref{sec:OptModel} will be copied for each time step~$t$, and we add
$t$ to the indices of variables and parameters, e.g., $\mf[a,t]$ will be the flow on
arc~$a$ at time $t$ and similarly for $\mfpos[a,t]$, $\mfneg[a,t]$,
$\tend[a,t]$, and $\tstart[v,t]$. Then the storages will be included as
follows.

Let $\estorage \subset \edges$ be the set of storage (`capacitor') arcs in
the graph $\graph = (\nodes, \edges)$. We assume that, similarly to heating
arcs, each storage arc is always directed from the return to the supply
network part, i.e., $\estorage \subset \eret \times \esup$. In the
following, storage arcs are parallel to heating arcs. Note
that the flow direction of the storage arc depends on whether the storage
is charging ($\mf[a,t] < 0$) or discharging ($\mf[a,t] > 0$).

As mentioned before, in each time step $t \in T$, a heat storage can
either extract spare heat power from the network or feed the charged
heat into the network, which is equivalent to heating or cooling the
water within the storage. The heat power $\stpower[a,t]$ of a storage
arc $a = (u,v) \in \estorage$ is modeled analogously to the heat power
on demand and supplier arcs (see Constraints~\eqref{eq:tdemand},
\eqref{eq:theating} and also \cite{Planungshandbuch}):
\begin{align*}
  \stpower[a,t] = \heatcap\, \abs{\mf[a,t]}\, (\tstart[a,t] - \tend[a,t]), 
\end{align*}
where $\stpower[a,t] < 0$ means feeding energy into the network (discharging)
and $\stpower[a,t] > 0$ means charging the storage. Consequently, the output temperature of a
storage also depends on the flow $\mf[a,t]$ and the temperature
level of the storage has to be high enough to enable this.

Similar to the start temperature of pipe arcs, see~\eqref{eq:tstart}, the
start temperature $\tstart[a,t]$ of a storage arc $a = (u,v) \in \estorage$ can
be defined depending on the binary flow variables:
$\tstart[a,t] \define \tnode[u,t] \, \mfpos[a,t] + \tnode[v,t] \, \mfneg[a,t]$.
The binary variables are modeled as in Section~\ref{sec:FlowDirection}.

We assume that the storage has a maximal storage capacity $\stcap{a} > 0$
and that the storage is fully loaded in the first time step. The
requirement that a storage can only feed energy into the network that was
charged in previous time steps results in the following conditions:
\begin{align*}
  0 \leq \stcap{a} + \sum_{\ell=1}^t \stpower[a,\ell] \leq \stcap{a} \quad
  \forall a \in \estorage,\; t \in T.
\end{align*}
Here, $\sum_{\ell=1}^t \stpower[a,\ell]$ is the net amount of energy of the
storage~$a$ up to time~$t$. Together with the initial energy level of
$\stcap{a}$, it should not exceed the capacity and not become negative.

Furthermore, the storage has a pump that regulates the mass flow and
pressure. Its electric power $\pumppower[a,t]$ for storage
$a = (u,v) \in \estorage$ is modeled analogously to the pump of supplier
arcs in Section~\ref{sec:SupplierModel} with
\begin{align}\label{eq:ElectricPowerStorage}
  \pumppower[a,t] = \frac{1}{\pumpeff[a]\, \rho} \, (\pr[v,t] - \pr[u,t]) \, \mf[a,t].
\end{align}
Recall that $\rho$ is the density of water and $\pumpeff[a] \in [0,1]$ denotes the
efficiency of the pump. Note that during charging of the
storage~$a$, we have $\mf[a,t] < 0$. Moreover, we need $\pr[v,t] > \pr[u,t]$ for
the supply part to be able to transport the hot water towards the demand
arcs. Thus, in this case $\pumppower[a,t] < 0$. Because of flow
conservation, however, this will be compensated by the pump of a
corresponding parallel heating, which we assume to exist (see
Assumption~\ref{ass:Storage} below).

In addition, we require
\[
  \stpower[a,t]\, \mf[a,t] \leq 0,
\]
which means that flow from the return to the supply network corresponds to
a supply of heat to the network, while flow from the supply to the return
network corresponds to heat consumption.

Then the model from Section~\ref{sec:OptModel} is applied for each time
step and the constraints for the storage are added. Let
$\eleccost[a,t]$ be the costs for the electric power~$\pumppower[a,t]$
and $\heatcost[a,t]$ the costs for the heat~$\heatpower[a,t]$ at time
$t \in T$, while the heat power of the storage is free in every time step
$t \in T$. Thus, the objective function is:
\begin{align*}
  \sum_{t \in T} \Bigg( \sum_{a \in \esupplier} \big(\eleccost[a,t]\, \pumppower[a,t] - \heatcost[a,t]\, \heatpower[a,t] \big) + \sum_{a \in \estorage} \eleccost[a,t]\, \pumppower[a,t] \Bigg).
\end{align*}
A complete model is given below in Section~\ref{sec:DecompositionModels}
and the implemented version is given in Appendix~\ref{sec:ImplementedModelSto}.

\subsection{Decomposition Approach for the Parallel Case}
\label{sec:Decomposition}

Unfortunately, the coupling of the time steps through the storage makes
global optimization quite difficult, even for a moderate number of time
steps. We therefore investigate the following decomposition approach. It
first ignores the storages, which decomposes the problem into $\card{T}$
many subproblems as discussed in the first part of this paper.
Then the charging/discharging schedule of the storages is
computed. The overall scheme is:
\begin{enumerate}[itemsep=0ex,parsep=0ex,topsep=0.5ex]
\item Solve $\card{T}$ independent optimization problems without storages.
\item Determine an optimal storage schedule by distributing the previously
  computed energy between storages and heatings.
\end{enumerate}

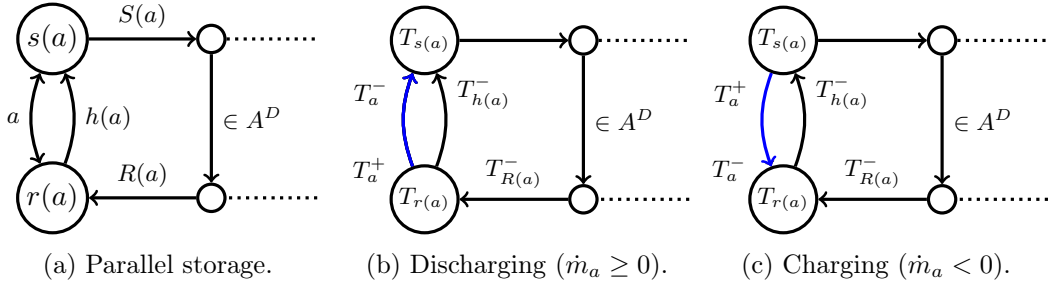
\begin{figure}
  \centering
  \begin{subfigure}[b]{0.3\textwidth}
    \centering
    \begin{tikzpicture}[scale=0.7]
      \node at (0,3) [circle,draw, very thick,inner sep=0.3ex] (va){$s(a)$};
      \node at (3,3) [circle,draw, very thick] (vb){};
      \node at (0,0) [circle,draw, very thick,inner sep=0.3ex] (ra){$r(a)$};
      \node at (3,0) [circle,draw, very thick] (rb){};
      \draw[->, very thick] (va)-- node[above]{\footnotesize $S(a)$}(vb);
      \draw[dotted, very thick] (vb)-- node[above]{}(5,3);
      \draw[dotted, very thick] (5,0)-- node[above]{}(rb);
      \draw[->, very thick] (rb)-- node[above]{\footnotesize $R(a)$}(ra);
      \draw[->, very thick] (ra) to[out=70,in=-70] node[right]{\footnotesize $h(a)$}(va);
      \draw[->, very thick] (vb)-- node[right]{\footnotesize $\in \edemand$}(rb);
      \draw[<->, very thick] (ra) to[out=110,in=-110] node[left]{\footnotesize $a$}(va);
    \end{tikzpicture}
    \caption{Parallel storage.}
    \label{fig:mf_parallel} 
  \end{subfigure}
  \begin{subfigure}[b]{0.3\textwidth}
    \centering
    \begin{tikzpicture}[scale=0.7]
      \node at (0,3) [circle,draw, very thick,inner sep=0.3ex] (va){\footnotesize $\tnode[s(a)]$};
      \node at (3,3) [circle,draw, very thick] (vb){};
      \node at (0,0) [circle,draw, very thick,inner sep=0.3ex] (ra){\footnotesize $\tnode[r(a)]$};
      \node at (3,0) [circle,draw, very thick] (rb){};
      \draw[->, very thick] (va)-- node[above]{}(vb);
      \draw[dotted, very thick] (vb)-- node[above]{}(5,3);
      \draw[dotted, very thick] (5,0)-- node[above]{}(rb);%
      \draw[->, very thick] (rb)-- node[above]{\footnotesize $\tend[R(a)]$}(ra);%
      \draw[->, very thick] (ra) to[out=70,in=-70] node[above right]{\footnotesize $\tend[h(a)]$}(va);%
      \draw[->, very thick] (vb)-- node[right]{\footnotesize $\in \edemand$}(rb);%
      \draw[->, very thick] (ra) to[out=110,in=-110] node[below left=10pt and 2pt]{\footnotesize $\tstart[a]$}(va);
      \draw[->, very thick] (ra) to[out=110,in=-110] node[above left=2pt]{\footnotesize $\tend[a]$}(va);
      \draw[->, very thick, blue] (ra) to[out=110,in=-110] node[left=2pt]{}(va);
    \end{tikzpicture}
    \caption{Discharging ($\mf[a] \geq 0$).}
    \label{fig:tmix_discharge}
  \end{subfigure}
  \begin{subfigure}[b]{0.3\textwidth}
    \centering
    \begin{tikzpicture}[scale=0.7]
      \node at (0,3) [circle,draw, very thick,inner sep=0.3ex] (va){\footnotesize $\tnode[s(a)]$};
      \node at (3,3) [circle,draw, very thick] (vb){};
      \node at (0,0) [circle,draw, very thick,inner sep=0.3ex] (ra){\footnotesize $\tnode[r(a)]$};
      \node at (3,0) [circle,draw, very thick] (rb){};
      \draw[->, very thick] (va)-- node[above]{}(vb);
      \draw[dotted, very thick] (vb)-- node[above]{}(5,3);
      \draw[dotted, very thick] (5,0)-- node[above]{}(rb);%
      \draw[->, very thick] (rb)-- node[above]{\footnotesize $\tend[R(a)]$}(ra);%
      \draw[->, very thick] (ra) to[out=70,in=-70] node[above right]{\footnotesize $\tend[h(a)]$}(va);%
      \draw[->, very thick] (vb)-- node[right]{\footnotesize $\in \edemand$}(rb);%
      \draw[-] (ra) to[out=110,in=-110] node[above left=2pt]{\footnotesize $\tstart[a]$}(va);
      \draw[<-] (ra) to[out=110,in=-110] node[below left=10pt and 2pt]{\footnotesize $\tend[a]$}(va);
      \draw[<-, very thick, blue] (ra) to[out=110,in=-110] node[below=2pt]{}(va);
    \end{tikzpicture}
    \caption{Charging ($\mf[a] < 0$).}
    \label{fig:tmix_charge}
  \end{subfigure}
  \caption{Scheme of temperature mixing for a parallel storage arc at time
    step $t \in T$ (dropped from notation).}
  \label{fig:4_tmix_parallel}
\end{figure}

To analyze conditions under which this approach produces optimal solutions, we need several
assumptions, see Figure~\ref{fig:mf_parallel} for an illustration of the setting:
\begin{assumption}\label{ass:Storage}\
  \begin{enumerate}[i),topsep=0.5ex,itemsep=0ex,parsep=0ex,leftmargin=4.5ex]
  \item There is a constant $\pumpeff[]$ such that for all
    $a \in \esupplier \cup \estorage$:
    $\pumpeff[a] = \pumpeff[]$.\label{ass:Storage:eff}
  \item Each storage $a \in \estorage$ is parallel to a unique heating arc
    $h(a) \in \esupplier$ and connects node~$r(a)$ in the return part to
    node $s(a)$ in the supply part. Its generated heat is not restricted, e.g.,
    because $\powerbound[h(a)] = -\infty$.
    \label{ass:Storage:parallel}
  \item For each heating arc there is a unique parallel storage arc.\label{ass:Storage:pair}
  \item For each storage $a \in \estorage$, there is a unique arc $R(a)$
    connecting the return part to $r(a)$ and a unique arc $S(a)$ connecting
    $s(a)$ to the supply part, both with positive flow
    direction; see Figure~\ref{fig:mf_parallel}.\label{ass:Storage:connection}
  \item For each $t \in T$ and $a \in \estorage$, we have
    $\eleccost[a,t] = \eleccost[h(a),t]$, i.e., the costs of the heating and storage
    pump are equal.\label{ass:Storage:cost}
  \end{enumerate}
\end{assumption}

Condition~\ref{ass:Storage:eff} implies that the pumps of
heatings and storages are of the same type. Condition~\ref{ass:Storage:parallel} is
motivated by real-life heating networks in which often there is a
centralized common storage near the heat generation source. This can, for
example, arise from a combined heat and power plant coupled with a storage,
but also decentralized storage locations could be very
advantageous~\cite{decentralStorage}. Their supplied heat needs to be
unrestricted or at least large enough in order to be able to simultaneously
charge the storage and supply the network. Condition~\ref{ass:Storage:pair}
is needed for the proofs, but is without loss of generality, since one can
add a parallel storage of capacity~0 if needed.
Condition~\ref{ass:Storage:connection} is essentially without loss of
generality, because the heatings and storages are usually connected by
unique arcs to the return and supply part of the network. The positive flow
direction assumption poses the restriction that a storage can only
be charged from the corresponding heating, but not from the network,
because the latter would require a negative flow on $S(a)$; this condition
is needed for showing optimality of the approach. Note that a storage can
still discharge into the network. Condition~\ref{ass:Storage:cost} implies
that pump costs for the storage and heating are equal.

In conclusion, these assumptions are more or less realistic. In any case,
we need them to prove optimality of the decomposition approach below,
which, however, can often also be applied as a heuristic.

\subsection{Models of the Decomposition Approach}\label{sec:DecompositionModels}

The overall optimization problem can be written as follows:
\begin{subequations}\label{opt:optModel}
  \begin{align}
    \min \; & \mathmbox{\sum_{t \in T} \bigg( \sum_{a \in \esupplier}  \big( \eleccost[a,t]\, \pumppower[a,t] - \heatcost[a,t]\, \heatpower[a,t] \big) + \sum_{a \in \estorage} \eleccost[a,t]\, \pumppower[a,t]\bigg)} &&\\
  \text{s.t.} \;
  & \sum_{a \in \delta^{+}(v)} \mf[a,t] - \sum_{a \in \delta^-(v)} \mf[a,t] = 0 && \forall\, v \in \nodes, t \in T,\label{eq:optModelFirst}\\
  & \text{TM}_{v,t}(\mf[],\mfpos[],\mfneg[],\tend[],\tnode[v,t]) && \forall v \in \nodes, t \in T,\\
  & \pr[u,t] - \pr[v,t] - \zeta_a \, \mf[a,t] \abs{\mf[a,t]} + \rho g \, (\height[u] - \height[v]) = 0 && \forall\, a \in \epipe, t \in T,\\
  & \pr[u,t] - \pr[v,t] \geq \plossmin && \forall\, a = (u,v) \in \edemand, t \in T,\\
  & \pumppower[a,t] = \frac{1}{\pumpeff[a]\, \rho} \, (\pr[v,t] - \pr[u,t]) \, \mf[a,t] && \forall\, a = (u,v) \in \esupplier \cup \estorage, t \in T,\\
  & \abs{\mf[a,t]}\, \tend[a,t] = \abs{\mf[a,t]}\, \tstart[a,t] - \tfrac{\heatloss[a,t] \, \length[a]}{\heatcap} \,  (\tstart[a,t] - \tamb) && \forall\, a \in \epipe, t \in T,\\
  & \heatcap \, \mf[a,t] \, (\tnode[u,t] - \tfix) = \heatdemand[a,t] && \forall\, a = (u,v) \in \edemand, t \in T,\\
  & \heatpower[a,t] - \heatcap \, \mf[a,t] \, (\tnode[u,t] - \tend[a,t]) = 0 && \forall\, a = (u,v) \in \esupplier, t \in T,\\
  & \stpower[a,t] -\heatcap \abs{\mf[a,t]} (\tstart[a,t] - \tend[a,t] ) = 0 && \forall\, a \in \estorage, t \in T,\\
  & 0 \leq \stcap{a} + \sum_{\ell=1}^t \stpower[a,\ell] \leq \stcap{a} && \forall\, a \in \estorage, t \in T,\\
  & \stpower[a,t]\, \mf[a,t] \leq 0 && \forall\, a \in \estorage, t \in T,\\
  & \tstart[a,t] - \tnode[a,t] \, \mfpos[a,t] - \tnode[a,t] \, \mfneg[a,t] = 0 && \forall\, a \in \estorage, t \in T,\\
  & \mfpos[a,t] + \mfneg[a,t] = 1 && \forall\, a \in \epipe \cup \estorage, t \in T,\\
  & \pr[\bar{v},t] = \pfix && \forall\, t \in T,\\
  & -\mubound[a,t] \, \mfneg[a,t] \leq \mf[a,t] \leq \mubound[a,t] \, \mfpos[a,t] && \forall\, a \in \epipe \cup \estorage, t \in T,\\
  & \tnode[v,t] \geq \tamb && \forall\, v \in \nodes, t \in T,\\
  & \tend[a,t] \geq \tamb && \forall\, a \in \edges, t \in T,\\
  & \tend[a,t] = \tfix && \forall\, a \in \edemand, t \in T,\\
  & \tend[a,t] \leq \tubound{a} && \forall\, a \in \esupplier, t \in T,\\
  & \pumppower[a,t] \geq 0 && \forall\, a \in \esupplier, t \in T,\\
  & \heatpower[a,t] \leq 0 && \forall\, a \in \esupplier, t \in T,\\  
  & \mf[a,t] \geq 0 && \forall\, a \in \edemand \cup \esupplier, t \in T,\\
  & \mfpos[a,t],\; \mfneg[a,t] \in \{0,1\} && \forall\, a \in \epipe \cup \estorage, t \in T.\label{eq:optModelLast}
  \end{align}
\end{subequations}
Here, $\text{TM}_{v,t}$ is the time-indexed version of~\eqref{eq:TemperatureMixing}.

We then get the following two optimization models for the decomposition
approach. The first arises from~\eqref{opt:optModel} by removing (or fixing
to 0) all variables corresponding to some $a \in \estorage$:
\begin{align}\label{opt:decomp1}\tag{D$^\prime$}
  \min \Big\{ & \sum_{t \in T} \sum_{a \in \esupplier} \eleccost[a,t]\, \pumppower[a,t] - \heatcost[a,t]\,\heatpower[a,t] \suchthat
                \text{\eqref{eq:optModelFirst}--\eqref{eq:optModelLast}},\;
                \mf[a,t] = \stpower[a,t] = \pumppower[a,t] = 0 \;\forall\,
                a \in \estorage\Big\},
\end{align}
which is equivalent to minimizing the sum of the objectives over copies of
the original model. As indicated earlier, this problem decomposes into
$\card{T}$ independent problems.

Consider a feasible solution $\mathcal{D}'$ of~\eqref{opt:decomp1} and
denote its values of variables by $(\cdot)'$, e.g., $\mf[a,t]'$. We
then obtain the second model:
\begin{equation}\label{opt:decomp2}\tag{D$^{\prime\prime}$}
  \begin{aligned}
    \min \quad & \sum_{t \in T} \bigg( \sum_{a \in \esupplier} \big( \eleccost[a,t]\, \pumppower[a,t] - \heatcost[a,t]\, \heatpower[a,t] \big) \, + \sum_{a \in \estorage} \eleccost[a,t]\, \pumppower[a,t]\bigg)\\
    \text{s.t.} \quad & \stpower[a,t] = \heatcap\, \abs{\mf[a,t]}\, (\tstart[a,t] - \tend[a,t]) && \forall a \in \estorage, t \in T, \\
    & \heatpower[h(a),t] = \heatcap\, \mf[h(a),t] (\tstart[h(a),t] - \tend[h(a),t]) && \forall a \in \estorage, t \in T, \\
    & \pumppower[a,t] + \pumppower[h(a),t] = \tfrac{1}{\pumpeff[]\, \rho} (\pr[s(a),t]' - \pr[r(a),t]') \, (\mf[a,t] + \mf[h(a),t]) && \forall a \in \estorage, t \in T, \\ 
    & 0 \leq \stcap{a} + \sum_{\ell=1}^t \stpower[a,\ell] \leq \stcap{a} && \forall a \in \estorage, t \in T, \\
    & \stpower[a,t]\, \mf[a,t] \leq 0 && \forall a \in \estorage, t \in T, \\
    & \eleccost[a,t] (\pumppower[a,t] + \pumppower[h(a),t]) - \heatcost[h(a),t] (\stpower[a,t] + \heatpower[h(a),t])\\
    & \qquad = \eleccost[h(a),t]\, \pumppower[h(a),t]' - \heatcost[h(a),t]\, \heatpower[h(a),t]' && \forall a \in \estorage, t \in T, \\
    & \mf[a,t] + \mf[h(a),t] = \mf[h(a),t]' && \forall a \in \estorage, t \in T,\\
    & \text{TM}_{v,t}(\mf[],\mfpos[],\mfneg[],\tend[],\tnode[v,t]) && \forall v \in \{r(a),s(a)\}, t \in T,\\
    & \tstart[a,t] \define \tnode[u,t] \, \mfpos[a,t] + \tnode[v,t] \, \mfneg[a,t],&& \forall a = (u,v) \in \estorage, t \in T,\\
    & \mfpos[a,t] + \mfneg[a,t] = 1,&& \forall a = (u,v) \in \estorage, t \in T,\\
    & -\mubound[a] \, \mfneg[a,t] \leq \mf[a,t] \leq \mubound[a] \, \mfpos[a,t],&&\forall a = (u,v) \in \estorage, t \in T,\\
    & \mfpos[a,t],\; \mfneg[a,t] \in \{0,1\}&&\forall a = (u,v) \in \estorage, t \in T.
  \end{aligned}
\end{equation}

Note that using~\eqref{eq:prpump}, \eqref{eq:ElectricPowerStorage} and
Assumption~\ref{ass:Storage}, the electric power for every storage
$a \in \estorage$ and time step $t \in T$ can be reformulated to
$\pumppower[a,t] + \pumppower[h(a),t] = \frac{1}{\pumpeff[]\, \rho}
(\pr[s(a),t] - \pr[r(a),t]) \, (\mf[a,t] + \mf[h(a),t])$, which is used in
the formulation of~\eqref{opt:decomp2}, Moreover, we remark that we couple
the heat energy of the storage over the time periods, but not the
temperature of the storage.


\subsection{Optimality of Decomposition Approach}\label{sec:DecompositionOptimality}

We will show that first solving \eqref{opt:decomp1} and then
\eqref{opt:decomp2} leads to an optimal solution for
Problem~\eqref{opt:optModel}, under Assumption~\ref{ass:Storage}. To this
end, we show that we get a feasible solution of
Problem~\eqref{opt:optModel} if we have feasible solutions for
\eqref{opt:decomp1} and \eqref{opt:decomp2}.  Then we show that every
feasible solution of~\eqref{opt:optModel} can be transferred to solutions
of \eqref{opt:decomp1} and \eqref{opt:decomp2} of the same value. Thus,
there does not exist a better solution for Problem~\eqref{opt:optModel}.

When we speak about solutions $\mathcal{D}'$ of~\eqref{opt:decomp1}, this refers to the following set of variables:
\[
  (\mf[a,t], \mfpos[a,t], \mfneg[a,t], \pr[u,t], \pr[v,t], \tnode[u,t], \tnode[v,t], \tstart[a,t], \tend[a,t], \pumppower[a,t], \heatpower[a,t])_{a = (u,v) \in \edges \setminus \estorage, t \in T}
\]
and a solution $\mathcal{D}''$ of~\eqref{opt:decomp2} uses the following variables
\[
  (\mf[a,t], \mfpos[a,t], \mfneg[a,t], \tnode[u,t], \tnode[v,t], \tstart[a,t], \tend[a,t], \pumppower[a,t], \heatpower[a,t],\stpower[a,t])_{a = (u,v) \in \esupplier, t \in T}
  \cup (\pumppower[a,t], \stpower[a,t])_{a \in \estorage, t \in T}.
\]

\begin{lemma}\label{lem:feasibilityParallel}
  Let $\mathcal{D}'$ and $\mathcal{D}''$ be feasible solutions
  of~\eqref{opt:decomp1} and~\eqref{opt:decomp2}, respectively, where
  $\mathcal{D}'$ is used as input to~\eqref{opt:decomp2}.  Then the
  solution $\mathcal{D}' \cup \mathcal{D}''$, where the values of
  $\mathcal{D}''$ overwrite those of~$\mathcal{D}'$, is a feasible
  solution of~\eqref{opt:optModel}.
\end{lemma}

\begin{proof}
  In $\mathcal{D}' \cup \mathcal{D}''$ all variables
  of~\eqref{opt:optModel} receive a value. The only variables
  appearing both in $\mathcal{D}'$ and $\mathcal{D}''$ are those
  incident to heating arcs $a \in \esupplier$. Except for these
  variables, the solutions are compatible with each other by
  construction of Problem~\eqref{opt:decomp2}; here, the constraints
  are satisfied via~\eqref{opt:decomp1}.  Moreover, the variables
  incident to heating arcs satisfy the constraints, because the values
  of $\mathcal{D}''$ are used.
\end{proof}

We now show that if we have a feasible solution of~\eqref{opt:optModel}, we
can construct a feasible solution for~\eqref{opt:decomp1} with a particular
structure.

\begin{lemma}\label{lem:SolutionTransferD1}
  Let $\mathcal{D}^\star$ be a feasible solution
  for~\eqref{opt:optModel}. Then under Assumption~\ref{ass:Storage}, there
  exists a feasible solution $\mathcal{D}'$ of~\eqref{opt:decomp1} such
  that
  \[
    \pumppower[h(a),t]' = \pumppower[a,t]^\star + \pumppower[h(a),t]^\star,\quad
    \heatpower[h(a),t]' = \stpower[a,t]^\star + \heatpower[h(a),t]^\star\quad
    \forall a \in \estorage, t \in T,
  \]
  where `$\mbox{\;}^\star$' and `$\mbox{\;}'$' indicate values of
  $\mathcal{D}^\star$ and $\mathcal{D}'$, respectively.
\end{lemma}

\begin{proof}
  We consider a fixed time step $t \in T$ and skip the index $t$ in the
  following. Moreover, we fix a storage arc $a \in \estorage$. We then
  abbreviate $r \define r(a)$, $s \define s(a)$, and $R \define R(a)$.

  The solution
  $\mathcal{D}' = (\mf[a]', \mfposprime[a], \mfnegprime[a], \pr[u]',
  \pr[v]', \tnode[u]', \tnode[v]', \tstartprime[a],
  \tendprime[a], \pumppower[a]', \heatpower[a]')_{a = (u,v) \in \edges
    \setminus \estorage}$ for~\eqref{opt:decomp1} is constructed
  such that all variables are copied from the given solution
  $\mathcal{D}^\star$, except the ones for $a$, $h(a)$, $r$, and $s$, where
  we define
  \begin{align*}
    & \pumppower[h(a)]' \define \pumppower[a]^\star + \pumppower[h(a)]^\star, \quad
    \heatpower[h(a)]' \define \stpower[a]^\star + \heatpower[h(a)]^\star, \quad
    \mf[h(a)]' \define \mf[a]^\star + \mf[h(a)]^\star,\\
    & \tnode[s]'\define \tnode[s]^\star,\quad \tnode[r]' \define \tendstar[R],\quad \tstartprime[h(a)] = \tnode[r]',
      \quad \tendprime[h(a)] \define \tnode[s]^\star,\quad \pr[r]' \define \pr[r]^\star, \quad \pr[s]' \define \pr[s]^\star.
  \end{align*}
  Note that we use `unboundedness' of the supplied heat of $h(a)$
  (Assumption~\ref{ass:Storage}~\ref{ass:Storage:parallel}).

  Since~\eqref{opt:decomp1} contains the same constraints
  as~\eqref{opt:optModel}, feasibility of all constraints away from~$a$ is
  clear. We now consider the constraints around $a$.

  Note that flow conservation is fulfilled by definition, since it is
  satisfied for $\mathcal{D}^\star$ and
  $\mf[h(a)]' = \mf[a]^\star + \mf[h(a)]^\star$ holds. Moreover, because
  there is exactly one unique incoming and outgoing
  arc by Assumption~\ref{ass:Storage}~\ref{ass:Storage:connection}, no
  temperature mixing at $r$ and $s$ occurs. Thus, the definition of the
  temperatures in $\mathcal{D}'$ is feasible.

  The next step is to show that the equations for $\heatpower[h(a)]'$ and
  $\pumppower[h(a)]'$ in~\eqref{opt:decomp1} are satisfied.  For the
  electric power $\pumppower[h(a)]'$, we obtain by feasibility of
  $\mathcal{D}^\star$, Assumption~\ref{ass:Storage}~\ref{ass:Storage:eff},
  and $\pr[r]' = \pr[r]^\star$, $\pr[s]' = \pr[s]^\star$:
  \begin{align*}
    \pumppower[h(a)]' & = \pumppower[a]^\star + \pumppower[h(a)]^\star
    = \tfrac{1}{\pumpeff[] \rho} (\pr[s]^\star - \pr[r]^\star) \, \mf[a]^\star + \tfrac{1}{\pumpeff[] \rho} (\pr[s]^\star - \pr[r]^\star) \, \mf[h(a)]^\star \\
                     & = \tfrac{1}{\pumpeff[] \rho} (\pr[s]' - \pr[r]') \, (\mf[a]^\star + \mf[h(a)]^\star)
                     = \tfrac{1}{\pumpeff[] \rho} (\pr[s]' - \pr[r]') \, \mf[h(a)]',
  \end{align*} 
  which shows that the equation for $\pumppower[h(a)]'$ is fulfilled.

  For the constraints on $\heatpower[h(a)]'$, we differentiate between a
  charging and a discharging storage in the given solution
  of~\eqref{opt:optModel}. In Figure~\ref{fig:tmix_discharge}, the
  temperature variables for a discharging storage, i.e.,
  $\stpower[a]^\star \leq 0$ and therefore $\mf[a]^\star \geq 0$, are shown. The
  mass flow directions are indicated by the arrows. Figure~\ref{fig:tmix_charge}
  shows the setting for a charging storage, i.e., $\stpower[a]^\star > 0$
  and therefore $\mf[a]^\star < 0$.  If the storage is charged, temperature
  mixing occurs at node $r$, while if the storage is discharged
  temperature mixing occurs at $s$. Thus, the solution $\mathcal{D}^\star$
  satisfies:
  \begin{align}
    \tnode[s]^\star =
    \frac{\mfposstar[a]\, \tendstar[a]\, \mf[a]^\star + \tendstar[h(a)]\, \mf[h(a)]^\star}{\mfposstar[a]\, \mf[a]^\star + \mf[h(a)]^\star},
    \qquad
    \tnode[r]^\star =
    \frac{\tendstar[R]\, \mf[R]^\star - \mfnegstar[a]\, \tendstar[a]\, \mf[a]^\star}{\mf[R]^\star - \mfnegstar[a]\, \mf[a]^\star}.
    \label{eq:tmix}
  \end{align}
  Moreover, we have
  \begin{align*}
    \heatpower[h(a)]^\star = \heatcap\, \mf[h(a)]^\star (\tstartstar[h(a)] - \tendstar[h(a)] ),\quad
    \stpower[a]^\star = \heatcap\, \abs{\mf[a]^\star} (\tstartstar[a] - \tendstar[a]).
  \end{align*}

  For $\mf[a]^\star \geq 0$, we can assume w.l.o.g.\ that
  $\mfposstar[a] = 1$, $\mfnegstar[a] = 0$. Thus,
  $\tnode[r]^\star = \tendstar[R] = \tnode[r]'$ and
  \begin{align*}
    \tfrac{1}{\heatcap} \heatpower[h(a)]' & = \tfrac{1}{\heatcap} \heatpower[h(a)]^\star + \tfrac{1}{\heatcap} \stpower[a]^\star\\
                     & =  \mf[h(a)]^\star (\tstartstar[h(a)] - \tendstar[h(a)]) +  \abs{\mf[a]^\star} (\tstartstar[a] - \tendstar[a])\\
                     & =  \mf[h(a)]^\star (\tnode[r]^\star - \tendstar[h(a)]) +  \mf[a]^\star (\tnode[r]^\star - \tendstar[a])\\
                     & =  (\mf[h(a)]^\star + \mf[a]^\star) \tnode[r]^\star -  (\mf[h(a)]^\star \tendstar[h(a)] + \mf[a]^\star \tendstar[a])\\
                     & \overset{\eqref{eq:tmix}}{=}  (\mf[h(a)]^\star + \mf[a]^\star) \tnode[r]^\star -  (\mf[h(a)]^\star + \mf[a]^\star) \tnode[s]^\star\\
                     & =  (\mf[h(a)]^\star + \mf[a]^\star) (\tnode[r]^\star - \tnode[s]^\star)\\
                     & \overset{\tnode[s]^\star = \tendprime[h(a)]}{=} \mf[h(a)]' (\tnode[r]' - \tendprime[h(a)]).
  \end{align*}
  Thus, the constraint for $\heatpower[h(a)]'$ is fulfilled.

  For $\mf[a]^\star < 0$ ($\mfposstar[a] = 0$, $\mfnegstar[a] = 1$),
  we have $\tnode[s]^\star = \tendstar[h(a)]$ because
  of~\eqref{eq:tmix}. Then using flow conservation (FC)
  $\mf[R]^\star = \mf[h(a)]^\star + \mf[a]^\star$ at $r$ yields:
  \begin{align*}
    \tfrac{1}{\heatcap} \heatpower[h(a)]' & = \tfrac{1}{\heatcap} \heatpower[h(a)]^\star + \tfrac{1}{\heatcap} \stpower[a]^\star\\
                     & =  \mf[h(a)]^\star (\tstartstar[h(a)] - \tendstar[h(a)]) +  \abs{\mf[a]^\star} (\tstartstar[a] - \tendstar[a])\\
                     & \overset{\mf[a]^\star \leq 0}{=}  \mf[h(a)]^\star (\tnode[r]^\star - \tendstar[h(a)]) -  \mf[a]^\star (\tnode[s]^\star - \tendstar[a])\\
                     & \overset{\tendstar[h(a)] = \tnode[s]^\star}{=} \mf[h(a)]^\star (\tnode[r]^\star - \tnode[s]^\star) -  \mf[a]^\star (\tnode[s]^\star - \tendstar[a])\\
                     & = \mf[h(a)]^\star \tnode[r]^\star - (\mf[h(a)]^\star + \mf[a]^\star) \tnode[s]^\star + \mf[a]^\star \tendstar[a]\\
                     & \overset{\text{FC}}{=}  (\mf[R]^\star - \mf[a]^\star) \tnode[r]^\star - \mf[R]^\star \tnode[s]^\star + \mf[a]^\star \tendstar[a]\\
                     & \overset{\eqref{eq:tmix}}{=} (\tendstar[R]\, \mf[R]^\star - \tendstar[a]\, \mf[a]^\star) - \mf[R]^\star \tnode[s]^\star + \mf[a]^\star \tendstar[a]\\
                     & = \mf[R]^\star (\tendstar[R] - \tnode[s]^\star)\\
                     & \overset{\text{FC}}{=} (\mf[a]^\star + \mf[h(a)]^\star) (\tendstar[R] - \tnode[s]^\star)\\
                     & = \mf[h(a)]' (\tnode[r]' - \tendprime[h(a)]).
  \end{align*}
  Thus, again the constraint for $\heatpower[h(a)]'$ is fulfilled.
\end{proof}

\begin{remark}
  Using Assumption~\ref{ass:Storage}~\ref{ass:Storage:cost}, the solution $\mathcal{D}'$
  constructed in Lemma~\ref{lem:SolutionTransferD1} has an objective
  value of
  \begin{align*}
    & \sum_{t \in T} \sum_{a \in \esupplier} \big( \eleccost[a,t]\, \pumppower[a,t]' - \heatcost[a,t]\, \heatpower[a,t]' \big)\\
    = & \sum_{t \in T} \sum_{a \in \estorage} \eleccost[a,t] (\pumppower[a,t]^\star + \pumppower[h(a),t]^\star) - \heatcost[h(a),t] (\stpower[a,t]^\star + \heatpower[h(a),t]^\star)\\
    = & \sum_{t \in T} \bigg( \sum_{a \in \esupplier} (\eleccost[a,t]\, \pumppower[a,t]^\star - \heatcost[a,t]\, \heatpower[a,t]^\star) + \sum_{a \in \estorage} \eleccost[a,t]\, \pumppower[a,t]^\star\bigg) - \sum_{t \in T} \sum_{a \in \estorage} \heatcost[h(a),t]\, \stpower[a,t]^\star,
  \end{align*}
  which is the objective of the given solution $\mathcal{D}^\star$
  for~\eqref{opt:optModel} plus the storage energy net cost
  $- \sum_{t \in T} \sum_{a \in \estorage} \heatcost[h(a),t]\, \stpower[a,t]^\star$ (recall that
  $\stpower[a,t]^\star$ is negative when charging).
\end{remark}

\begin{lemma}\label{lem:SolutionTransferD2}
  Let Assumption~\ref{ass:Storage} hold and let $\mathcal{D}^\star$ be a
  feasible solution for~\eqref{opt:optModel}. Let $\mathcal{D}'$ be any
  feasible solution of~\eqref{opt:decomp1}. Using $\mathcal{D}'$ as input,
  there exists a feasible solution $\mathcal{D}''$ for~\eqref{opt:decomp2}
  which satisfies:
  \begin{equation}\label{eq:SolutionTransferD2}
    \stpower[a,t]'' = \stpower[a,t]^\star,\quad
    \heatpower[h(a),t]'' = \heatpower[h(a),t]' - \stpower[a,t]^\star,\quad
    \pumppower[a,t]'' + \pumppower[h(a),t]'' = \pumppower[h(a)]'
    \quad \forall a \in \estorage, t \in T.
  \end{equation}
  Here, the superscripts `$\mbox{\;}^\star$', `$\mbox{\;}^\prime$', and
  `$\mbox{\;}^{\prime\prime}$' mark values of $\mathcal{D}^\star$,
  $\mathcal{D}'$, and $\mathcal{D}''$, respectively.
\end{lemma} 

\begin{proof}
  We will construct a feasible solution $\mathcal{D}''$
  for~\eqref{opt:decomp2} that
  satisfies~\eqref{eq:SolutionTransferD2}. We again consider a fixed
  time step $t \in T$ and storage arc $a \in \estorage$. We then skip
  the index $t$ in the following and abbreviate $r \define r(a)$,
  $s \define s(a)$, and $R \define R(a)$.

  We will use the following conditions, which arise from~\eqref{opt:decomp2} and~\eqref{eq:SolutionTransferD2}:
  \begin{align}
    & \heatcap\, \abs{\mf[a]''} (\tstartprimeprime[a] - \tendprimeprime[a]) = \stpower[a]^\star,\label{eq:Transfer2:stpower}\\
    & \heatcap\, \mf[h(a)]'' (\tstartprimeprime[h(a)] - \tendprimeprime[h(a)]) = \heatpower[h(a)]' - \stpower[a]^\star,\label{eq:Transfer2:heatpower}\\
    & \mf[a]'' + \mf[h(a)]'' = \mf[h(a)]'.\label{eq:TransferD2:flow}
  \end{align}

  The solution $\mathcal{D}''$ can be defined as follows. Let
  $\mf[a]'' \define \min\{\mf[a]^\star,
  \smash{\mf[h(a)]'}\}$. Then~\eqref{eq:TransferD2:flow} implies
  $\mf[h(a)]'' \define \mf[h(a)]' - \mf[a]'' \geq 0$. The remaining
  values depend on the flow direction of $\mf[a]''$.

  If $\mf[a]'' > 0$ and $\mf[h(a)]'' > 0$ (see
  Figure~\ref{fig:tmix_discharge}), we define:
  \begin{equation}\label{eq:Transfer2:positive}
    \begin{aligned}
      & \tnode[r]'' = \tstartprimeprime[a] = \tstartprimeprime[h(a)] \define \tnode[r]',\quad
      \tnode[s]'' \define \tnode[s]',\\
      & \tendprimeprime[a] \define \tstartprimeprime[a] - \frac{1}{\heatcap\, \mf[a]'' } \stpower[a]^\star,\quad
      \tendprimeprime[h(a)] \define \tstartprimeprime[h(a)] - \frac{1}{\heatcap\, \mf[h(a)]''} (\heatpower[h(a)]' - \stpower[a]^\star).
    \end{aligned}
  \end{equation}
  Note that the definitions should be performed in the given
  order. Then by definition,
  \eqref{eq:Transfer2:stpower} and~\eqref{eq:Transfer2:heatpower} are
  satisfied.
  
  If $\mf[a]'' < 0$ and $\mf[h(a)]'' > 0$ (see
  Figure~\ref{fig:tmix_charge}), we define:
  \begin{equation}\label{eq:Transfer2:negative}
    \begin{aligned}
      & \tnode[s]'' = \tstartprimeprime[a] = \tendprimeprime[h(a)] \define \tnode[s]',\quad
      \tendprimeprime[a] \define \tstartprimeprime[a] + \frac{1}{\heatcap\, \mf[a]'' } \stpower[a]^\star,\\
      & \tstartprimeprime[h(a)] \define \tendprimeprime[h(a)] + \frac{1}{\heatcap\, \mf[h(a)]''} (\heatpower[h(a)]' - \stpower[a]^\star),\quad
        \tnode[r]'' \define \frac{\mf[h(a)]' \tnode[r]' - \mf[a]'' \tendprimeprime[a]}{\mf[h(a)]''}.
    \end{aligned}
  \end{equation}
  Again, \eqref{eq:Transfer2:stpower} and \eqref{eq:Transfer2:heatpower} are satisfied.

  If $\mf[a]'' = 0$ (i.e., $\mf[a]^\star = 0$), we have
  $\stpower[a]^\star = 0$. In this case, we can use all values from
  $\mathcal{D}'$ ($\tstartprimeprime[a]$ and $\tendprimeprime[a]$ are
  arbitrary) and be feasible for~\eqref{opt:decomp2}.

  If $\smash{\mf[h(a)]''} = 0$, we have
  $\mf[a]'' = \smash{\mf[h(a)]'} \geq 0$. We can assume that
  $\mf[a]'' > 0$, because $\mf[a]'' = 0$ has already been
  discussed. We then use the definitions
  in~\eqref{eq:Transfer2:positive}, except that
  $\tendprimeprime[a] \define \tnode[s]'$ (note that
  $\smash{\tstartprimeprime[h(a)]}$ and $\tendprimeprime[h(a)]$ are
  arbitrary). Then~\eqref{eq:Transfer2:stpower} becomes, using
  feasibility of $\mathcal{D}'$:
  \[
    \stpower[a]^\star = \heatcap\, \abs{\mf[a]''} (\tstartprimeprime[a] -
    \tendprimeprime[a]) = \heatcap\, \abs{\mf[a]''} (\tnode[r]' -
    \tnode[s]') = \heatcap\, \abs{\mf[h(a)]'} (\tstartprime[a] -
    \tendprime[a]) = \heatpower[h(a)]'.
  \]
  Thus, $\stpower[a]^\star = \smash{\heatpower[h(a)]'}$
  and~\eqref{eq:Transfer2:heatpower} is satisfied.  The same holds for
  all other constraints.  This concludes the definition of the solution
  $\mathcal{D}''$.

  We now claim that the produced solution is feasible
  for~\eqref{opt:decomp2}.  First note that by definition, the signs of
  $\mf[a]''$ and $\mf[a]^\star$ are always equal, since $\mf[h(a)]'\ge 0$, which implies that
  $\smash{\stpower[a]''\, \mf[a]'' \leq 0}$ holds.

  We next consider temperature mixing. If $\mf[a]'' > 0$, mixing
  happens at node $s$:
  \begin{align*}
    \tnode[s]'' & = \frac{\mf[a]'' \tendprimeprime[a] + \mf[h(a)]'' \tendprimeprime[h(a)]}{\mf[a]'' + \mf[h(a)]''}\\
    & = \frac{(\mf[a]'' \tnode[r]'' - \frac{1}{\heatcap} \stpower[a]^\star) + (\mf[h(a)]''\tnode[r]'' - \frac{1}{\heatcap} (\heatpower[h(a)]' - \stpower[a]^\star))}{\mf[h(a)]'}\\
    & = \frac{(\mf[a]'' + \mf[h(a)]'') \tnode[r]'' - \frac{1}{\heatcap} \heatpower[h(a)]'}{\mf[h(a)]'}\\
    & = \frac{\mf[h(a)]' \tnode[r]' - \frac{1}{\heatcap} \heatpower[h(a)]'}{\mf[h(a)]'},
  \end{align*}
  which is equivalent to
  \[
    \mf[h(a)]' (\tnode[s]' - \tnode[r]') = - \tfrac{1}{\heatcap} \heatpower[h(a)]',
  \]
  which holds, because $\mathcal{D}'$ is feasible.

  Similarly, if $\mf[a]'' < 0$, temperature mixing occurs at $r$,
  where flow conservation implies that
  $\mf[a]'' + \mf[h(a)]'' - \mf[R]'' = 0$, i.e.,
  $\mf[R]'' - \mf[a]'' = \mf[h(a)]''$. Because we do not change the
  flows and temperatures on arc $R$, we get
  $\mf[R]'' = \mf[R]' = \mf[h(a)]'$ and
  $\tendprimeprime[R] = \tendprime[R] = \tnode[r]'$. This yields:
  \[
    \tnode[r]'' = \frac{\mf[R]'' \tendprimeprime[R] - \mf[a]'' \tendprimeprime[a]}{\mf[R]'' - \mf[a]''}
    = \frac{\mf[h(a)]' \tnode[r]' - \mf[a]'' \tendprimeprime[a]}{\mf[h(a)]''},
  \]
  which is true by definition~\eqref{eq:Transfer2:negative}.

  Furthermore, the following constraint of~\eqref{opt:decomp2} is
  directly implied by~\eqref{eq:SolutionTransferD2} and
  Assumption~\ref{ass:Storage}, implying
  $\eleccost[a] = \eleccost[h(a)]$:
  \begin{align*}
    \eleccost[a] (\pumppower[a]'' + \pumppower[h(a)]'') - \heatcost[h(a)] (\stpower[a]'' + \heatpower[h(a)]'')
    = \eleccost[h(a)]\, \pumppower[h(a)]' - \heatcost[h(a)]\, \heatpower[h(a)]'.
  \end{align*}
  Since $\stpower[a]'' = \stpower[a]^\star$ and
  $\stpower[a]^\star$ is feasible for~\eqref{opt:optModel}, we obtain for
  all $t \in T$:
  \begin{align*}
    0 \leq \stcap{a} + \sum_{\ell=1}^t \stpower[a,\ell]'' \leq \stcap{a}.
  \end{align*}
  For the electric power, we get
  \begin{align*}
    \pumppower[a]'' + \pumppower[h(a)]''
    = \pumppower[h(a)]' = \frac{1}{\pumpeff[] \rho} (\pr[s]' - \pr[r]') \, \mf[h(a)]'
    = \frac{1}{\pumpeff[] \rho} (\pr[s]' - \pr[r]') \, (\mf[h(a)]'' + \mf[a]'').
  \end{align*}
  In total, this proves that $\mathcal{D}''$ is a feasible solution for~\eqref{opt:decomp2}.
\end{proof}

\begin{lemma}\label{lem:DecompOptimal}
  Let Assumption~\ref{ass:Storage} hold and let $\mathcal{D}^\star$ be a
  feasible solution for~\eqref{opt:optModel}. Let $\mathcal{D}'$ be an
  optimal solution of~\eqref{opt:decomp1}. If this solution is used as
  input, there exists a solution $\mathcal{D}''$ of~\eqref{opt:decomp2},
  whose objective value is at least as good as $\mathcal{D}^\star$.
\end{lemma}

\begin{proof}
  Consider the given feasible solution $\mathcal{D}^\star$ with values
  $\stpower[a,t]^\star$, $\smash{\heatpower[h(a),t]^\star}$,
  $\pumppower[a,t]^\star$, and $\smash{\pumppower[h(a),t]^\star}$ and
  denote by $\heatpower[h(a),t]'$ and $\pumppower[h(a),t]'$ the components
  of the optimal solution of~\eqref{opt:decomp1} for all $t \in T$ and
  $a \in \estorage$.
  
  One of the following two cases occurs.
  \begin{align}
     \sum_{t \in T} \sum_{a \in \estorage}\eleccost[a,t] (\pumppower[a,t]^\star + \pumppower[h(a),t]^\star) - \heatcost[h(a),t] (\stpower[a,t]^\star + \heatpower[h(a),t]^\star) 
    < \sum_{t \in T} \sum_{a \in \estorage}\eleccost[h(a),t]\, \pumppower[h(a),t]' - \heatcost[h(a),t]\, \heatpower[h(a),t]',\label{eq:DecompCostLess}
   \end{align}
   or
   \begin{align}
     \sum_{t \in T} \sum_{a \in \estorage} \eleccost[a,t] (\pumppower[a,t]^\star + \pumppower[h(a),t]^\star) - \heatcost[h(a),t] (\stpower[a,t]^\star + \heatpower[h(a),t]^\star)
    \ge \sum_{t \in T} \sum_{a \in \estorage}\eleccost[h(a),t]\, \pumppower[h(a),t]' - \heatcost[h(a),t]\,
    \heatpower[h(a),t]'.\label{eq:DecompCostGreatereq}
  \end{align}

  Consider the first case, i.e., \eqref{eq:DecompCostLess} holds.  By
  Lemma~\ref{lem:SolutionTransferD1} there exists a feasible solution
  of~\eqref{opt:decomp1} with values $\heatpower[h(a),t]^\diamond$ and
  $\pumppower[h(a),t]^\diamond$ such that
  $\pumppower[h(a),t]^\diamond = \pumppower[a,t]^\star +
  \pumppower[h(a),t]^\star$ and
  $\heatpower[h(a),t]^\diamond = \stpower[a,t]^\star +
  \heatpower[h(a),t]^\star$ for all $a \in \estorage, t \in T$.  By
  Assumption~\ref{ass:Storage}~\ref{ass:Storage:parallel}, \ref{ass:Storage:pair} and~\ref{ass:Storage:cost},
  it follows that
  \begin{align*}
    \sum_{t \in T} \sum_{a \in \estorage} \eleccost[h(a),t]\, \pumppower[h(a),t]^\diamond - \heatcost[h(a),t]\, \heatpower[h(a),t]^\diamond
    & = \sum_{t \in T} \sum_{a \in \estorage} \eleccost[a,t]\, (\pumppower[a,t]^\star + \pumppower[h(a),t]^\star) - \heatcost[h(a),t] (\stpower[a,t]^\star + \heatpower[h(a),t]^\star)\\
    & < \sum_{t \in T} \sum_{a \in \estorage} \eleccost[h(a),t]\, \pumppower[h(a),t]' - \heatcost[h(a),t]\, \heatpower[h(a),t]'.
  \end{align*}
  Thus, $\mathcal{D}'$ cannot be optimal -- a contradiction.

  It remains to treat the case~\eqref{eq:DecompCostGreatereq}. By
  Lemma~\ref{lem:SolutionTransferD2}, we can choose $\mathcal{D}''$ such that
  \[
    \stpower[a,t]'' = \stpower[a,t]^\star,\quad
    \heatpower[h(a),t]'' = \heatpower[h(a),t]' - \stpower[a,t]^\star,\quad
    \pumppower[a,t]'' + \pumppower[h(a),t]'' = \pumppower[h(a)]'\quad \forall a \in \estorage, t \in T.
  \]
  Then, by Assumption~\ref{ass:Storage}~\ref{ass:Storage:cost} the
  objective of $\mathcal{D}''$ satisfies:
  \begin{align*}
    & \sum_{t \in T} \sum_{a \in \estorage} \eleccost[h(a),t]\, \pumppower[h(a),t]'' - \heatcost[h(a),t]\, \heatpower[h(a),t]'' + \eleccost[a,t]\, \pumppower[a,t]''\\
    = & \sum_{t \in T} \sum_{a \in \estorage} \eleccost[h(a),t]\, (\pumppower[h(a),t]'' + \pumppower[a,t]'') - \heatcost[h(a),t]\, \heatpower[h(a),t]''\\
    = & \sum_{t \in T} \sum_{a \in \estorage} \eleccost[h(a),t]\, \pumppower[h(a),t]' - \heatcost[h(a),t]\, \heatpower[h(a),t]' + \heatcost[h(a),t]\, \stpower[a,t]^\star\\
    \overset{\eqref{eq:DecompCostGreatereq}}{\le} & \sum_{t \in T} \sum_{a \in \estorage} \eleccost[h(a),t]\, (\pumppower[a,t]^\star + \pumppower[h(a),t]^\star) - \heatcost[h(a),t]\, (\stpower[a,t]^\star + \heatpower[h(a),t]^\star) + \heatcost[h(a),t]\, \stpower[a,t]^\star\\
    = & \sum_{t \in T} \sum_{a \in \estorage} \eleccost[h(a),t]\, (\pumppower[a,t]^\star + \pumppower[h(a),t]^\star) - \heatcost[h(a),t]\, \heatpower[h(a),t]^\star.
  \end{align*}
  This means that the objective of $\mathcal{D}''$ is at least as small as the one of $\mathcal{D}^\star$.
\end{proof}

\begin{corollary}\label{cor:DecompOptimal}
  Let Assumption~\ref{ass:Storage} hold and let $\mathcal{D}^\star$ be a
  feasible solution for~\eqref{opt:optModel}. Let $\mathcal{D}'$ be an
  optimal solution of~\eqref{opt:decomp1} and $\mathcal{D}''$ be an
  optimal solution of~\eqref{opt:decomp2}, using $\mathcal{D}'$ as
  input. Then $\mathcal{D}' \cup \mathcal{D}''$ (in the sense of
  Lemma~\ref{lem:feasibilityParallel}) is an optimal solution
  for~\eqref{opt:optModel}.
\end{corollary}

\begin{proof}
  Lemma~\ref{lem:feasibilityParallel} shows that
  $\mathcal{D}' \cup \mathcal{D}''$ is feasible
  for~\eqref{opt:optModel}. Lemma~\ref{lem:DecompOptimal} shows that for
  every feasible solution of~\eqref{opt:optModel} there exists a feasible
  solution for Problem \eqref{opt:decomp2} with objective value at least as
  good. Therefore, $\mathcal{D}' \cup \mathcal{D}''$ is optimal
  for~\eqref{opt:optModel}.
\end{proof}

\begin{remark}
  If Assumption~\ref{ass:Storage} do not hold, then the decomposition
  approach might not work, i.e., \eqref{opt:decomp2} can be infeasible or
  the approach produces a suboptimal solution. However, it can often
  produce a heuristic solution.
\end{remark}

\subsection{Numerical Results for the Decomposition Approach}

\begin{figure}
  \centering
  \begin{tikzpicture}
    \node at (0,3) [circle,draw, very thick] (va){};
    \node at (5,3) [circle,draw, very thick] (vb){};
    \node at (0,0) [circle,draw, very thick] (ra){};
    \node at (5,0) [circle,draw, very thick] (rb){};
    \draw[->, very thick] (va)-- node[above]{}(vb);
    \draw[->, very thick] (rb)-- node[above]{}(ra);
    \draw[->, very thick] (ra) to [out=70,in=-70] node[right]{$a^S$}(va);
    \draw[->, very thick] (vb)-- node[left]{$a^D$}(rb);
    \draw[->, very thick] (ra) to [out=110,in=-110] node[left]{$a^C$}(va);
  \end{tikzpicture}
  \caption{Example network with a storage $a^C$, heating $a^S$, and demand $a^D$.}
  \label{fig:storage_example}
\end{figure}
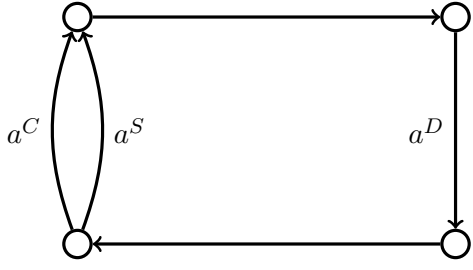

In this section, we illustrate the decomposition on an example. Because the
solution of the full problem with storages, i.e.,
Problem~\eqref{opt:optModel}, takes a very long time, we use an extremely
simple network, see Figure~\ref{fig:storage_example}.

\begin{table}[tb]
  \centering
  \caption{Parameters for decomposition example.}
  \label{tab:decomp_example_param}
  \begin{tabular}{@{}lllr@{}}\toprule
    $t$ & heat cost $\heatcost[]$ [\euromwh] & electricity cost $\eleccost[]$ [\euromwh] & demand\\\midrule
    1 & 0.072     & 0.000072 & 80\\
    2 & 0.0599684 & 0.0000599684 & 48\\
    3 & 0.0612    & 0.0000612 & 60\\
    4 & 0.0571856 & 0.0000571856 & 40\\\bottomrule
  \end{tabular}
\end{table}

We consider four time steps, i.e., $T = \{1, \dots, 4\}$ with changing costs
and heat demand given in Table~\ref{tab:decomp_example_param}. The demand
values should roughly reflect changes within one day, i.e., the different
time steps imitate morning, noon, evening, and night. In the morning the
heat demand and also the costs are the highest, while at night the demand
and the costs are lowest. Further, we assume a maximal supplier heat of
\SI{300}{\kW} and a storage capacity of $100$. Recall that the storage is
fully loaded in the first time step.

\begin{table}[tb]
  \caption{Comparison of decomposition approach.}
  \label{tab:decomp_results}
  \centering
  \begin{tabular}{@{}lrrr@{}}
    \toprule 
    & Original~\eqref{opt:optModel} & \eqref{opt:decomp1} & \eqref{opt:decomp2} \\ \midrule
    Solving time [s] & 3600 & 0.26 & 0.10 \\
    Best objective  & \num{403.637} & \num{410.842} & \num{403.718} \\
    \bottomrule
  \end{tabular}
\end{table}

We solve the original problem \eqref{opt:optModel} and compare it with
first solving Problem~\eqref{opt:decomp1} and
then~\eqref{opt:decomp2}.  We set the time limit again to 3600
seconds.  Table~\ref{tab:decomp_results} shows the solving time and
the objective values of the three models. As expected, Problem
\eqref{opt:decomp1} and \eqref{opt:decomp2} are easy to solve and only
need a few seconds until the optimal solution is found. In contrast,
the original Problem \eqref{opt:optModel} runs into the time
limit. After \SI{10.95}{\second}, a feasible solution is found, but
the program was not able to produce a reasonable dual bound and thus
cannot prove optimality.

Considering the objective values one can see that the
value of~\eqref{opt:optModel} coincides with the values found by the
decomposition approach within reasonable accuracy. As expected, the value
of~\eqref{opt:decomp1} is worse, since the storage is not used.

\section{Conclusion}
\label{sec:Conclusion}

This paper has demonstrated that global optimization of district
heating networks for larger inter-meshed networks with several
suppliers is possible in practice. However, the larger the number of
cycles or suppliers, the more complicated a solution of the stationary
problems gets. For the case with several time steps, a decomposition
method was essential. These results suggest immediate open challenges
for the future: the solution of even larger problems and instances
with storages that do not meet the assumptions for the decomposition
approach. An evaluation for the usage of the decomposition method as a
heuristic would also be interesting.

\section*{Acknowledgements}

Partially supported by the German Federal Ministry for Economic Affairs and Climate Action (BMWK) under project number 03EN3012A.

\begin{small}
   \bibliographystyle{IEEEtranSmod}
   \bibliography{heatnetpaper.bib}
\end{small}

\appendix

\section{Implemented Model}
\label{sec:ImplementedModel}

\begin{align*}
  \min\quad & \sum_{a \in \esupplier} (\pumppower[a] - \heatpower[a]), \\
  \text{s.t.} & \sum_{a \in \delta^{+}(v)} \mf[a] - \sum_{a \in \delta^-(v)} \mf[a] = 0 && \forall\, v \in \nodes,\\
  & \sum_{a \in \delta^-(v)} \mfp[a]\, \tnode[v] - \sum_{a \in \delta^+(v)} \mfn[a]\, \tnode[v]\nonumber\\
  & \quad = \sum_{a \in \delta^-(v)} \mfp[a]\, \tend[a] - \sum_{a \in \delta^+(v)} \mfn[a]\, \tend[a] && \forall\, v \in \nodes,\\
  & \pr[u] - \pr[v] = \zeta_{a} \, \mf[a] \abs{\mf[a]} + \rho g \, (\height[u] - \height[v]), && \forall\, a = (u,v) \in \epipe,\\
  & \pr[u] - \pr[v] \geq \plossmin && \forall\, a = (u,v) \in \edemand,\\
  & \pumppower[a] = \frac{1}{\pumpeff[a]\, \rho} \, (\pr[v] - \pr[u]) \, \mf[a] && \forall\, a = (u,v) \in \esupplier,\\
  & \mfp[a]\, \tend[a] - \mfn[a]\, \tend[a] \\
  & \quad = \mfp[a]\, \tnode[u] + \mfn[a]\, \tnode[v] - \tfrac{ \heatloss[a] \, \length[a]}{\heatcap} \, (\mfpos[a]\, \tnode[u] + \mfneg[a]\, \tnode[v] - \tamb) && \forall\, a = (u,v) \in \epipe,\\
  & \heatcap \, \mf[a] \, (\tnode[u] - \tfix) = \heatdemand[a] && \forall\, a = (u,v) \in \edemand,\\
  & \heatpower[a] - \heatcap \, \mf[a] \, (\tnode[u] - \tend[a]) = 0 && \forall\, a = (u,v) \in \esupplier,\\
  & \mfpos[a] + \mfneg[a] = 1 && \forall\, a \in \epipe,\\
  & \pr[\bar{v}] = \pfix, && \\
  & -\mubound[a] \, \mfneg[a] \leq \mf[a] \leq \mubound[a] \, \mfpos[a] && \forall\, a \in \epipe,\\
  & \mfp[a] = \mf[a]\, \mfpos[a] && \forall\, a \in \edges\\
  & \mfn[a] = \mf[a]\, \mfneg[a] && \forall\, a \in \edges\\
  & \tnode[v] \geq \tamb && \forall\, v \in \nodes,\\
  & \tend[a] \geq \tamb && \forall\, a \in \epipe,\\
  & \tend[a] = \tfix && \forall a \in \edemand,\\
  & \tend[a] \leq \tubound{a} && \forall\, a \in \esupplier,\\
  & \pumppower[a] \geq 0 && \forall\, a \in \esupplier,\\
  & \powerbound[a] \leq \heatpower[a] \leq 0 && \forall\, a \in \esupplier,\\
  & \mf[a] \geq 0 && \forall\, a \in \edemand \cup \esupplier,\\
  & \mfpos[a],\; \mfneg[a] \in \{0,1\} && \forall\, a \in \epipe.
\end{align*}

\section{Implemented Model with Storages}
\label{sec:ImplementedModelSto}

\begin{align*}
  \min \quad &   \sum_{t \in T} \Bigg( \sum_{a \in \esupplier} (\eleccost[a,t]\, \pumppower[a,t] - \heatcost[a,t]\, \heatpower[a,t]) + \sum_{a \in \estorage} \eleccost[a,t]\, \pumppower[a,t] \Bigg)\\
  \text{s.t.} & \sum_{a \in \delta^{+}(v)} \mf[a,t] - \sum_{a \in \delta^-(v)} \mf[a,t] = 0 && \forall\, v \in \nodes, t \in T,\\
  & \sum_{a \in \delta^-(v)} \mfp[a,t]\, \tnode[v,t] - \sum_{a \in \delta^+(v)} \mfn[a,t]\, \tnode[v,t]\\
  & \quad = \sum_{a \in \delta^-(v)} \mfp[a,t]\, \tend[a,t] - \sum_{a \in \delta^+(v)} \mfn[a,t]\, \tend[a,t] && \forall\, v \in \nodes, t \in T,\\
  & \pr[u,t] - \pr[v,t] - \zeta_a \, \mf[a,t] \abs{\mf[a,t]} + \rho g \, (\height[u] - \height[v]) = 0 && \forall a \in \epipe, t \in T,\\
  & \pr[u,t] - \pr[v] \geq \plossmin && \forall\, a = (u,v) \in \edemand, t \in T,\\
  & \pumppower[a,t] = \frac{1}{\pumpeff[a]\, \rho} \, (\pr[v,t] - \pr[u,t]) \, \mf[a,t] && \forall\, a = (u,v) \in \esupplier \cup \estorage, t \in T,\\
  & \mfp[a,t]\, \tend[a,t] - \mfn[a,t]\, \tend[a,t] \\
  & \quad = \mfp[a,t]\, \tnode[u,t] + \mfn[a,t]\, \tnode[v,t] - \tfrac{ \heatloss[a] \, \length[a]}{\heatcap} \, (\mfpos[a,t]\, \tnode[u,t] + \mfneg[a,t]\, \tnode[v,t] - \tamb) && \forall\, a = (u,v) \in \epipe, t \in T,\\
  & \heatcap \, \mf[a,t] \, (\tnode[u,t] - \tfix) = \heatdemand[a,t] && \forall\, a = (u,v) \in \edemand, t \in T,\\
  & \heatpower[a,t] - \heatcap \, \mf[a,t] \, (\tnode[u,t] - \tend[a,t]) = 0 && \forall\, a = (u,v) \in \esupplier, t \in T,\\
  & \stpower[a,t] -\heatcap \abs{\mf[a,t]} (\tstart[a,t] - \tend[a,t] ) = 0, && \forall a \in \estorage, t \in T,\\
  & 0 \leq \stcap{a} + \sum_{\ell=1}^t \stpower[a,\ell] \leq \stcap{a} && \forall a \in \estorage, t \in T,\\
  & \tstart[a,t] - \tnode[a,t] \, \mfpos[a,t] - \tnode[a,t] \, \mfneg[a,t] = 0 && \forall a \in \estorage, t \in T,\\
  & \mfpos[a,t] + \mfneg[a,t] = 1 && \forall\, a \in \epipe \cup \estorage, t \in T,\\
  & \pr[\bar{v},t] = \pfix && \forall t \in T,\\
  & -\mubound[a,t] \, \mfneg[a,t] \leq \mf[a,t] \leq \mubound[a,t] \, \mfpos[a,t] && \forall\, a \in \epipe \cup \estorage, t \in T,\\
  & \mfp[a,t] = \mf[a,t]\, \mfpos[a,t] && \forall\, a \in \edges, t \in T,\\
  & \mfn[a,t] = \mf[a,t]\, \mfneg[a,t] && \forall\, a \in \edges, t \in T,\\
  & \tnode[v,t] \geq \tamb && \forall v \in \nodes, t \in T,\\
  & \tend[a,t] \geq \tamb && \forall a \in \edges, t \in T,\\
  & \tend[a,t] = \tfix && \forall a \in \edemand, t \in T,\\
  & \tend[a,t] \leq \tubound{a} && \forall\, a \in \esupplier, t \in T,\\
  & \pumppower[a,t] \geq 0 && \forall\, a \in \esupplier, t \in T,\\
  & \powerbound[a,t] \leq \heatpower[a] \leq 0 && \forall\, a \in \esupplier, t \in T,\\
  & \stpower[a,t] \, \mf[a,t] \leq 0 && \forall a \in \estorage, t \in T,\\
  & \mf[a,t] \geq 0 && \forall\, a \in \edemand \cup \esupplier, t \in T,\\
  & \mfpos[a,t],\; \mfneg[a,t] \in \{0,1\} && \forall\, a \in \epipe \cup \estorage, t \in T.\\
\end{align*}

\end{document}